\documentclass[12pt, a4paper,  reqno, english]{amsart}
\usepackage[all]{xy}
\usepackage[margin=0.9in]{geometry} 
\usepackage{amssymb,amsmath,amsthm}
\usepackage{comment}
\usepackage{enumerate}
\numberwithin{equation}{section}

\newcommand{\mb}{\mathbb}

\newtheorem{theorem}{Theorem}
\newtheorem{proposition}{Proposition}
\newtheorem{lemma}{Lemma}
\newtheorem{corollary}{Corollary}

\theoremstyle{definition}

\renewcommand{\mod}[1]{\hspace{-2.9mm}\pmod{#1}}

\newcommand{\ra}{\rightarrow}
\renewcommand{\ss}{\substack}

\newcommand{\mc}{\mathcal}

\newcommand{\ben}{\begin{enumerate}}
\newcommand{\een}{\end{enumerate}}
\newcommand{\eit}{\begin{itemize}}
\newcommand{\beq}{\begin{equation}}
\newcommand{\eeq}{\end{equation}}

\newcommand{\eps}{\epsilon}

\newcommand{\asum}{\sideset{}{^{\ast}}\sum}

\renewcommand{\leq}{\leqslant}
\renewcommand{\geq}{\geqslant}
\newcommand{\llf}{\left\lfloor}
\newcommand{\rrf}{\right\rfloor}

\newcommand{\sg}{\sigma}

\usepackage{color}
\definecolor{red}{rgb}{1,0,0}

\begin{document}

\title{Three conjectures about character sums}
\author{Andrew Granville}

 \address{D{\'e}partment  de Math{\'e}matiques et Statistique,   Universit{\'e} de Montr{\'e}al, CP 6128 succ Centre-Ville, Montr{\'e}al, QC  H3C 3J7, Canada; and  Department of Mathematics, University College London, Gower Street, London WC1E 6BT, England.}
   \email{andrew@dms.umontreal.ca}  
	
\author{Alexander P. Mangerel}
\address{Department of Mathematical Sciences, Durham University, Stockton Road, Durham, DH1 3LE, England}
\email{smangerel@gmail.com}

 \thanks{We would like to thank Dimitris Koukoulopoulos and K. Soundararajan for helpful discussions.
 A.G.~is partially supported by grants from NSERC (Canada). Most of this paper was completed while A.M.~ was a CRM-ISM postdoctoral fellow at the Centre de Recherches Math\'{e}matiques.}  
\begin{abstract}
We establish that three well-known and rather different looking conjectures about Dirichlet characters and their (weighted) sums, (concerning the P\'{o}lya-Vinogradov theorem for maximal character sums,  the maximal admissible range in Burgess' estimate for short character sums, and upper bounds for $L(1,\chi)$ and $L(1+it,\chi)$) are more-or-less ``equivalent''.
We also obtain a new mean value theorem for logarithmically weighted sums of 1-bounded multiplicative functions.
\end{abstract}

\maketitle

\newcommand{\cbar}{\overline{\chi}}
\newcommand{\pbar}{\overline{\psi}}
\newcommand{\sumstar}{\sideset\and ^* \to \sum}

\section{Introduction}

\subsection{Conjectures about weighted sums of Dirichlet characters}
Let $\chi$ be a primitive character mod $q>1$ and
\[
S(\chi,N):= \sum_{n\leq N} \chi(n) \text{ for all } N\geq 1.
\]
The three most widely-used, unconditionally proved estimates about characters sums are:
\begin{itemize}
\item The \emph{P\'{o}lya-Vinogradov} theorem:
\[
M( \chi ) := \max_{1 \leq N\leq q} |S(\chi,N)| \leq c_1 \sqrt{q}\log q
\]
for some explicit $c_1>0$;
\item \emph{Burgess's theorem} (\cite{Bur}, \cite{Hil}): For $N\geq q^{c_2}$ (and $q$ \emph{cube-free})
\[
|S(\chi,N)|  =o( N)
\]
for any $c_2\geq \tfrac 14$; and
\item The Dirichlet $L$-function $L(s,\chi)$ at $s=1$ satisfies
\[
|L(1,\chi)|=\bigg|\sum_{n\geq 1} \frac{\chi(n)}n \bigg| \leq c_3\log q
\]
for some explicit $c_3>0$.  One can also show that for any fixed $T>0$, there exists a constant $c_3(T) >0$ such that if $t\in [-T,T]$ then
$|L(1+it,\chi)| \leq c_3(T)\log q$.
\end{itemize}
The Riemann Hypothesis for $L(s,\chi)$ implies that one can take any $c_1,c_2,c_3 >0$ but this has resisted unconditional proof.  One unlikely but currently plausible `obstruction' to establishing this unconditionally is the possibility  that  $\chi(p)=1$ for all primes $p\leq q^c$, in which case $c_1,c_2,c_3\gg c$, or indeed if $\chi$ is $1$-pretentious for the primes up to $q$.\footnote{``Pretentiousness'' will be defined in section 3.} 

Inspired by connections highlighted in \cite{BoGo, FrGo, Man} 
we show that improving any one of these bounds will, more-or-less, improve the others.

\begin{theorem} \label{thm: main} The following statements are equivalent:
\begin{itemize}
\item   There exists  $\kappa_1>0$ such that there are infinitely many primitive characters $\chi \pmod q$   for which  $M(\chi) \geq \kappa_1\sqrt{q}\log q$;
\item  There exists   $\kappa_3>0$ such that there are infinitely many odd primitive characters $\psi \pmod r$    for which 
$|L(1,\psi)|\geq \kappa_3\log r$.
\end{itemize}
\end{theorem}

 This follows from a more precise connection:

\begin{corollary} \label{cor: MandL} Suppose that $\chi$ is a primitive character mod $q$. We have
$M(\chi) \gg \sqrt{q}\log q$ if and only if  there exists a primitive character 
$\xi \pmod \ell$ with $\xi(-1)=-\chi(-1)$ and $\ell\ll 1$ for which $|L(1,\psi)|\gg \log q$, where $\psi$ is the primitive (odd) character that induces $\chi\bar\xi$.
\end{corollary}

In other words we prove that if $M(\chi) \gg \sqrt{q}\log q$ then $\chi$ is $\xi$-pretentious for some $\xi$ of bounded conductor, and we will also establish a converse theorem.

  Next we relate large $S(\chi, N)$-values with large $L(1+it,\chi)$-values:

\begin{theorem} \label{thm: main2} The following statements are equivalent:
\begin{itemize}
\item There exists  $\kappa_2>0$ such that there are infinitely many primitive characters $\chi \pmod q$  
for which there is an integer $N\in [q^{\kappa_2},q]$ such that
$|S(\chi,N)|  \geq \kappa_2 N$;
\item  There exists   $\kappa_3>0$ and $T>0$ such that there are infinitely many primitive characters $\chi \pmod q$ for which there exists $t\in [-T,T]$  such that  $|L(1+it,\chi)|\geq \kappa_3\log q$.
\end{itemize}
\end{theorem}

If we restrict attention here to characters of bounded order then one can take $t=0$.
The precise connection is given in the following result.

 \begin{proposition} \label{prop: LandB} Fix $c>0$. Let $\chi$ be a primitive  character mod $q$.
There exists $t\in \mathbb R$ with $|t|\ll 1$ for which $|L(1+it,\chi)| \geq c\log q$
if and only if there exist $\kappa = \kappa(c) > 0$ and $x\in [q^\kappa,q]$ for which $|S(\chi,x)| \gg_c  x$.
If $\chi^k=\chi_0$ for some $k\ll 1$  then we may take  $t=0$.
\end{proposition}

In other words we prove that if $|S(\chi,N)| \gg  N$ for some $N>q^\kappa$ and $\kappa > 0$ then $\chi$ is $n^{it}$-pretentious  for some bounded real number $t$, and we will also establish a converse theorem.

We can combine these results: If $M(\chi) \gg \sqrt{q}\log q$ and $|S(\chi,N)| \gg  N$ for some $N>q^\kappa$,
then $\chi$ is both $\xi$-pretentious and $n^{it}$-pretentious, which implies that $\xi$ is $n^{it}$-pretentious, where $\xi$ is a primitive character of bounded conductor. We will show that this implies  $\xi=1$ and $t=0$, so that 
$\chi$ is an odd character and $1$-pretentious for the primes up to $q$ (see Corollary \ref{cor: allLarge}). 
Therefore such a putative character is  the only obstruction to improving at least one of our three famous results unconditionally (that is, being able to take any $c_1>0$ in  the P\'{o}lya-Vinogradov theorem, or
being able to take any $c_2>0$ in Burgess's theorem, or being able to take any $c_3>0$ in bounds for  $L(1,\chi)$).

Other new results on this topic will be discussed in section \ref{sec: DifferentSums}.

\subsection{Logarithmic averages of multiplicative functions}  We prove our results on  sums of characters by viewing characters as examples of  multiplicative functions that take their values on the unit disk
$\mb{U} := \{z \in \mb{C} : |z| \leq 1\}$.   Hal\'asz's theorem, which we will discuss in detail below, bounds the mean value of $f(n)$ for $n$ up to $x$ in terms of how ``pretentious'' $f$ is. In particular, if $f$ is real-valued then
Hall and Tenenbaum \cite{HT} showed that 
 \begin{align}\label{eq:HTbd}
\sum_{n \leq x} f(n)  \ll  x\, e^{-\tau \mathbb{D}(f,1;x)^2 } \text{ where } 
\mathbb{D}(f,1;x)^2=\sum_{p\leq x} \frac{1-\text{Re}(f(p))}p,
\end{align}
and 
\[
\tau=0.3286\ldots=-\cos \theta \text{ where } \theta\in (0,\pi) \text{ satisfies }
\sin \theta - \theta \cos \theta =\frac \pi 2.
\]
 They gave an example where one attains equality in \eqref{eq:HTbd} (up to the inexplicit constant).  

We give an analogous result for logarithmic averages of the form $\sum_{n\leq x} f(n)/n$ though, as discovered in \cite{GSPret}, we do not need to restrict attention to real-valued $f$. Here we let $\lambda\in \mathbb R$ be such that 
\[
\int_0^{1} |e(\theta)-\lambda| d\theta = 2-\lambda,
\]
so that $\lambda=0.8221\dots$.

 \begin{proposition} \label{prop:lambdaOpt}
Let $f:\mb{N} \ra \mb{U}$ be a multiplicative function. \\
\noindent {\rm (a)}\ We have
\begin{equation} \label{eq:GS2b}
\sum_{n \leq x} \frac{f(n)}{n} \ll (\log x)(1+\mathbb D(f,1;x)^2) e^{ -\lambda \, \mathbb D(f,1;x)^2 } + \log\log x.
\end{equation} 
{\rm (b)}\ The exponent $\lambda$ is ``best possible'' in a result of this kind, since there exists
 a multiplicative function $f: \mb{N} \ra \mb{U}$ such that
\begin{equation} \label{eq:GS25}
\bigg|\sum_{n \leq x} \frac{f(n)}{n}\bigg| \asymp (\log x) e^{-\lambda \, \mb{D}(f,1;x)^2}.
\end{equation}
\end{proposition}

The $\lambda$ in the bound \eqref{eq:GS2b}  improves on the $\tfrac 12$ in the bound given in Lemma 4.3 of \cite{GSPret}.

We deduce Proposition \ref{prop:lambdaOpt}(b) from Theorem \ref{lem:asympForft} which establishes asymptotics for 
$\sum_{n \leq N} f(n)/n$ for a class of multiplicative functions $f:\mb{N} \ra \mb{U}$ for which 
$f(p)=g(\tau \log p)$ for each prime $p$ for some fixed small real $\tau$, where $g(t)$ is a $1$-periodic function with ``well-behaved'' Fourier coefficients.

More details, as well as other new results on this topic, will be discussed in section \ref{sec: HalaszBeyond}.

\section{Connections between different sums of characters} \label{sec: DifferentSums}

\subsection{Large character sums} \label{subsec LargeChisum}
P\'{o}lya gave the following Fourier expansion (see e.g., Lemma 1 of \cite{MV}) for character sums: for $\alpha \in [0,1)$
\begin{align} \label{eq_PFE}
\sum_{n \leq \alpha q} \chi(n) = \frac{ g(\chi) }{ 2\pi i } \sum_{1 \leq |n| \leq q} \frac{ \bar{\chi}(n) }{n} \bigg( 1-e(-n\alpha) \bigg) + O(\log q),
\end{align}
where $e(t) := e^{2\pi i t}$ for $t \in \mb{R}$ and $g(\chi) := \sum_{a \pmod{q}} \chi(a) e( \tfrac aq )$ is the Gauss sum. When $\chi$ is primitive we know that  $|g(\chi)|=\sqrt{q}$ and 
\begin{equation}\label{eq: L1,chi}
\sum_{1 \leq |n| \leq q} \frac{ \bar{\chi}(n) }{n} = (1-\chi(-1))L(1,\bar{\chi}) +o(1),
\end{equation}
and so to estimate the left-hand side of \eqref{eq_PFE} for any $\alpha$ we are left to estimate the sums 
\[
\sum_{n\leq q}  \frac{\bar{\chi}(n)e(\pm n\alpha)}n.
\]
 Note that \eqref{eq: L1,chi} is large if and only if $\chi$ is an odd character and $L(1,\chi)$ is large.

Fix $\tfrac 2\pi<\Delta<1$ and let $R_q:=\exp( \tfrac {(\log q)^\Delta }{\log\log q} )$.
For any $\alpha\in [0,1)$ we may obtain an approximation $|\alpha-\tfrac bm|<\tfrac 1{mR_q}$, with $(b,m)=1$ and $m\leq R_q$, by Dirichlet's theorem.

If $r_q:= (\log q)^{2-2\Delta}(\log\log q)^4<m\leq R_q$ then we say that $\alpha$ is on a \emph{minor arc}. By straightforward modifications to the proof of Lemma 6.1 of \cite{GSPret}, for such $\alpha$ we get
\begin{align}\label{eq: minorArc}
\sum_{n\leq q} \frac{\chi(n) e(n\alpha)}n \ll \log \log q + \frac{(\log r_q)^{3/2}}{\sqrt{r_q}} \log q + \log R_q = o( (\log q)^\Delta ).
\end{align}
If $m\leq r_q$ then we say that $\alpha$ is on a \emph{major arc}.  Let $N:=\min\{ q,\tfrac 1{|m\alpha -b|}\}$, 
so that $R_q\leq N\leq q$. By Lemma 6.2 of \cite{GSPret},
\begin{align} \label{eq: NewSum}
\sum_{1 \leq |n| \leq q} \frac{\chi(n) e(n\alpha)}n = \sum_{1 \leq |n| \leq  N} \frac{\chi(n) e(n\tfrac bm)}n+O(\log\log q).
\end{align}
Therefore, if $M(\chi)\gg \sqrt{q}(\log q)^\Delta$ then
either $\chi$ is odd and $L(1,\chi)\gg (\log q)^\Delta$, or
$M(\chi) = |S(\chi,\alpha q)|$ where $\alpha$ lies on a major arc.
The following proposition provides more detailed information in these cases.

 \begin{proposition}\label{prop: BoundsOnMchi}
Fix $\tfrac 2\pi < \Delta < 1$ and let $\chi$ be a character mod $q$. We have
\[ M(\chi) \gg \sqrt{q}(\log q)^{\Delta}\]
 if and only if there is a primitive character 
$\xi \pmod \ell$ with $\xi(-1)=-\chi(-1)$ and $\ell\leq  (\log q)^{2-2\Delta}(\log\log q)^4$ such that 
\[
\max_{1 \leq N \leq q} \bigg|\sum_{n \leq N}  \frac { (\chi\bar{\xi})(n) }n \bigg| \gg \frac{ \phi( \ell ) }{\sqrt{\ell } } (\log q)^{\Delta}.
\]

More precisely, in this case we have
\begin{equation}\label{eq: MchiAsymp}
M(\chi)  \sim \tau_{\chi,\xi} \cdot \frac{  \sqrt{q\ell}  }{\pi \phi(\ell)} \cdot  \max_{1\leq N\leq q} \bigg| \sum_{n\leq N} \frac{(\chi\bar{\xi})(n)}n \bigg|
\end{equation}
where $\tau_{\chi,1}\in [\tfrac 12,3]$ and $\tau_{\chi,\xi}=\max\{ 1,   |1-(\chi\bar{\xi})(2)| \}$ if $\xi \ne 1$.
\end{proposition}

Throughout, $N_q$  denotes an integer value of $N$ that maximizes the right-hand side of \eqref{eq: MchiAsymp}.

Using results from the next subsection, we will deduce Corollary \ref{cor: MandL} by showing that 
if $M(\chi) \gg \sqrt{q} \log q$ then for $\psi=\chi\bar{\xi}$ we have
\begin{align}\label{eq:qToNq}
|L(1,\psi)|\gg \log q \text{  if and only if  } \bigg| \sum_{n\leq N_q} \frac{\psi(n)}n \bigg|\gg \log q.
\end{align}
However, there is not necessarily  a correspondence between these two sums when they are slightly smaller.  
For example, if $\psi(p)=1$ for all $p\leq N:=\exp((\log q)^\tau)$ and 
$\psi(p)=-1$ for all $ \exp((\log q)^\tau)<p\leq q^\delta$ 
where   $\tfrac 12 < \tau<1$ and  $\delta>0$ is some small fixed constant, then assuming $\psi$ is non-exceptional (see \eqref{eq:L1Prod}),
\[
\bigg| \sum_{n\leq N} \frac{\psi(n)}n \bigg|\asymp (\log q)^{\tau} \text{ while }  |L(1,\psi)|\asymp \prod_{p \leq q} \bigg|1-\frac{\psi(p)}{p}\bigg|^{-1} \asymp (\log q)^{2\tau-1},
\]
which is much smaller. Moreover this (purported) example shows why \emph{we cannot assume} that $N_q=q$ and that the largest sum is $\sim |L(1,\psi)|$.

Following an idea from \cite{BoGo}, Proposition \ref{prop: BoundsOnMchi} has the following consequence for quadratic non-residues:\footnote{It is worth recalling that the real primitive characters are given by $1(\cdot)$ as well as $(\frac \cdot n)$ if $n>1$, $(\frac {2n} \cdot)$ and $(\frac {4n} \cdot)$ if $n\equiv 3 \pmod 4$, for each odd squarefree integer $n$.}

 \begin{corollary} \label{cor: QRBd} Fix $\Delta \in ( \tfrac 2\pi , 1)$. Let $n_q$ be the least quadratic non-residue modulo a prime  $q \equiv 3 \pmod{4}$  and suppose that $n_q\geq \exp( (\log q)^{\Delta})$.
 For any odd, squarefree integer $\ell\leq \big(\frac{ \log n_q } { (\log q)^{\Delta} }\big)^2$  we have
\[
M((\tfrac \cdot  {\ell q} )) \geq (c\tau-o(1)) \sqrt{q} \log n_q \text{ where } c=\tfrac{2(\sqrt{e}-1)}{ \pi}=0.41298\dots,
\]
$\tau=\frac 12$ if $\ell=1$, otherwise $\tau=1$ or $2$ depending on whether $q\equiv 7$ or $3 \pmod 8$, respectively.
 \end{corollary}

 \begin{proof} Let $\xi=(\tfrac  \cdot \ell)$ and 
 $\chi=(\tfrac \cdot  {\ell q} )$, so that $\chi \overline{\xi} = (\tfrac  \cdot q) 1_{(\cdot,\ell)=1}$ and 
  \[
 \sum_{n \leq x}  \frac { (\chi\bar{\xi})(n) }n = \sum_{\substack{n\leq x\\ (n,\ell)=1}} \frac{(\tfrac n q)}n
 \]
 Let $y=n_q-1$ and $N=x=y^w$ where $w=e^{1/2}$.
 A \emph{$y$-smooth integer} has all of its prime factors $\leq y$, and any $y$-smooth integer here  is a quadratic residue mod $q$. Therefore 
 \begin{align*}  
\sum_{\substack{n\leq x\\ (n,\ell)=1}}  \frac{(\tfrac n q)}n\geq \sum_{\substack{n\leq x\\ P(n)\leq y \\ (n,\ell) = 1}} \frac 1n - 
 \sum_{\substack{n\leq x\\ P(n)>y\\ (n,\ell)=1}} \frac 1n =2 \sum_{\substack{n\leq x\\ P(n)\leq y\\ (n,\ell)=1}} \frac 1n - 
\sum_{\substack{n\leq x\\ (n,\ell)=1}} \frac 1n
 \end{align*}
where $P(n)$ is the largest prime factor of $n$.  Therefore, since $1_{(n,\ell)=1}=\sum_{d|(n,\ell)} \mu(d) $, and as $\ell\leq y$ we deduce that
\begin{align} \label{eq: nqLowBd2}
\sum_{\substack{n\leq x\\ (n,\ell)=1}}  \frac{(\tfrac n q)}n\geq 
\sum_{d | \ell} \mu(d)  \bigg( 2 \sum_{\substack{n\leq x\\ d|n\\ P(n)\leq y}} \frac 1n  - \sum_{\substack{n\leq x\\ d|n}} \frac 1n \bigg) =
\sum_{d | \ell} \frac{\mu(d)}d  \bigg( 2 \sum_{\substack{m\leq x/d \\ P(n)\leq y}} \frac 1m  - \sum_{\substack{m\leq x/d}} \frac 1m \bigg) .
 \end{align}
Let  $\psi(x,y)$ be the number of $y$-smooth integers $\leq x$. It is well-known that $\psi(y^u,y)= y^u\rho(u)(1+O(\frac 1{\log y}))$ as $y\to \infty$ for bounded $u$, with $\rho(u)=1$ for $0\leq u\leq 1$ and $\rho(u)=1-\log u$ for $1\leq u\leq 2$.
Therefore by partial summation we have
\[
\sum_{\substack{n\leq x\\ P(n)\leq y}} \frac 1n = \int_{t=1}^x \frac{\psi(t,y)}{t^2}dt +O(1)
= \int_{u=0}^w \rho(u) du \cdot \log y +O(1),
\]
and
\[
\int_{u=0}^w \rho(u) du=\int_{u=0}^1 du+\int_{u=1}^w (1-\log u) du = \frac 32 w-1,
\]
so that 
\[
2 \sum_{\substack{n\leq x\\ P(n)\leq y}} \frac 1n - 
\sum_{\substack{n\leq x}} \frac 1n = (2w-2)\log y +O(1).
\]  
Since $d\leq \ell=y^{o(1)}$ we can use this in \eqref{eq: nqLowBd2} with $x$ replaced by $x/d$ to obtain
 \begin{align*}
\sum_{\substack{n\leq x\\ (n,\ell)=1}}  \frac{(\tfrac n q)}n &\geq 
\sum_{d | \ell} \frac{\mu(d)}d  ( (2w-2)\log y+O(1))\\
& = \frac{\phi(\ell)}{\ell}   (2w-2)\log y+O\bigg(\frac{\ell}{\phi(\ell)}\bigg) \sim \frac{\phi(\ell)}{\ell}   (2\sqrt{e}-2)\log n_q \\
& \gg \frac{\phi(\ell)}{\sqrt{\ell}}   (\log q)^\Delta
\end{align*}
using the hypothesis on the size of $\ell$. The hypothesis of the second part of Proposition \ref{prop: BoundsOnMchi} 
is therefore satisfied (with $q$ replaced by $\ell q$), and so by \eqref{eq: MchiAsymp} and the last displayed equation we have
\[
M((\tfrac \cdot  {\ell q} ))  \sim \tau_{\chi,\xi} \cdot \frac{  \sqrt{q} \ell }{\pi \phi(\ell)} \cdot  \max_{1\leq N\leq q} \bigg| \sum_{\substack{n\leq N\\ (n,\ell)=1}} \frac{(\tfrac n q)}n \bigg| \gtrsim
\frac{ 2(\sqrt{e}-1)   \tau_{\chi,\xi}  }{\pi  } \cdot    \sqrt{q} \log n_q
\]
where $\tau_{\chi,1}\in [\tfrac 12,3]$ and $\tau_{\chi,\xi}=\max\{ 1,   |1-(\frac 2q)| \}$ if $\xi \ne 1$.
     \end{proof}

 \section{Hal\'asz's Theorem and beyond}\label{sec: HalaszBeyond}
 
  For multiplicative functions $f,g : \mathbb{N} \rightarrow \mathbb{U}$ and $x \geq 2$, we define the \emph{pretentious distance}
$$
\mathbb{D}(f,g;x) := \bigg(\sum_{p \leq x } \frac{1-\text{Re}(f(p)\bar{g}(p))}{p}\bigg)^{1/2}.
$$
It is well-known that $\mathbb{D}$ satisfies the triangle inequality:
\begin{align}\label{eq: triIneq}
\mb{D}(f,h;x) \leq \mb{D}(f,g;x) + \mb{D}(g,h;x) \text{ for all } f,g,h : \mb{N} \ra \mb{U}.
\end{align}
With $2 \leq y \leq x$ we also write $\mathbb{D}(f,g;y,x)$ to work only with the primes in $(y,x]$.
We say that $f$ is \emph{$g$-pretentious} (for the primes up to $x$) if 
$\mb{D}(f,g;x)  \ll 1$; so if $f$ is $g$-pretentious then $f(p) \approx g(p)$ frequently for $p \leq x$.

 \subsection{Hal\'{a}sz's theorem}
For $T>0$, $x \geq 2$ and a multiplicative function $f: \mb{N} \ra \mb{U}$, we also define
$$
M(f;x,T) := \min_{|t| \leq T} \mathbb{D}(f,n^{it};x)^2.
$$
We let $t=t(f;x,T)$ be a real number in this range where we get equality.
Hal\'asz's Theorem (see e.g., \cite[Thm. 1]{GSDec}) states that if $1\leq T\leq \log x$ then, for $M=M(f;x,T)$,
\begin{align} \label{eq: HalThm}
\sum_{n \leq x} f(n)  \ll  (1+M)e^{-M}x + \frac{x}{T}.
\end{align}
If $f(n)=n^{it}$ with $|t| \leq T$ then $M=0$, which reflects the fact that 
$|\sum_{n \leq x}n^{it}|\sim x/|1+it|$.

Hal\'{a}sz's theorem shows that $\bigg|\sum_{n \leq x} f(n)\bigg|$ is $o(x)$ if $f$ is not $n^{it}$-pretentious for any $t\in \mb{R}$.
Elementary estimates for $\zeta(s)$ to the right of the $1$-line imply that 
 \begin{equation} \label{eq: Dfornit}
 \mathbb{D}(1,n^{it};x)^2=\begin{cases}
 \log(1+|t|\log x) + O(1) & \text{ if } |t|\leq 100;\\
 \log\log x +O(\log\log |t|)& \text{ if } |t|\geq 100.
 \end{cases}
 \end{equation}
This shows that $1$ is \emph{not} $n^{it}$-pretentious unless $|t| \ll \tfrac 1{\log x}$.
Therefore if  $|t_1|,|t_2| \leq  T := (\log x)^{O(1)}$ then   $n^{it_1}$ cannot be $n^{it_2}$-pretentious unless $|t_1-t_2|\log x  \ll 1$. 

If $t(f;x,T)$ is not unique, say $t_1$ and $t_2$ both yield equality above, then \eqref{eq: triIneq} implies that 
$\mb{D}(n^{it_1},n^{it_2};x) \leq \mb{D}(f,n^{it_1};x) + \mb{D}(f,n^{it_2};x)=2\mb{D}(f,n^{it_1};x)$. In particular if 
$f$ is  $n^{it_1}$-pretentious  then $f$ is not $n^{it_2}$-pretentious unless $t_2=t_1+O( \tfrac 1{\log x} )$.

\subsection{Hal\'asz-type bounds for  logarithmically weighted sums}

If $f$ is real-valued then we might expect that $t(f;x,T)=0$ but there are examples where this is not so
(which lead to the ``best possible examples'' in Hall and Tenenbaum's estimate \eqref{eq:HTbd}).
The examples  $f(n)=n^{it}$ show that there cannot be an upper bound in terms of $\mathbb{D}(f,1;x)$ alone for arbitrary $f : \mb{N} \ra \mb{U}$, though such bounds are given in \cite{HT} for $f$ belonging to certain restricted families of multiplicative functions (most importantly those of bounded order).

In this article we will need bounds for the logarithmically weighted sums $\sum_{n \leq x} f(n)/n$. 

\begin{proposition} \label{prop: logsum} Let $x \geq 3$. 
Let $f:\mb{N} \ra \mb{U}$ be a multiplicative function with 
 $M=M(f;x,1)$, and let $t \in [-1,1]$ minimize the expression
\[
\tau \mapsto \sum_{p \leq x} \frac{2-\text{\rm Re}( (1+f(p)) p^{-i\tau} ) }{p}, \quad \tau \in [-1,1].
\]
If  $|t| \leq \tfrac 1{\log x}$ then
 \begin{equation} \label{eq:GS.2}
\sum_{n \leq x} \frac{f(n)}{n} \ll (1+M)e^{-M}\log x + \log\log x. 
\end{equation} 
 If  $|t| \geq \tfrac 1{\log x}$ then
\begin{equation} \label{eq:GS2}
\sum_{n \leq x} \frac{f(n)}{n} \ll \frac 1{|t|}(1+M + \log(|t|\log x))e^{-M} + \log\log x.  
\end{equation} 
\end{proposition}
 
The bound \eqref{eq:GS2} improves upon Theorem 2.4 in \cite{Gold} and Theorem 1.4 in \cite{LamMan} whenever $|t|\gg \tfrac {\log\log x}{\log x}$ (though the latter can be used to replace $(1+M)e^{-M} $ by just $e^{-M}$). \\

\subsection{Deductions} Using Proposition \ref{prop: logsum} we now  prove Corollary \ref{cor: MandL}.

\begin{proof} [Proof of more than \eqref{eq:qToNq}]
Let $\psi = \chi\bar{\xi}$.
By the definition of $N_q$,
$$
\mathcal L:=\bigg|\sum_{n \leq N_q} \frac{\chi\bar{\xi}(n)}{n}\bigg| \geq \bigg|\sum_{n \leq q} \frac{\chi\bar{\xi}(n)}{n}\bigg| = |L(1,\psi)| + O(1)
$$
(see \eqref{eq:L1Trunc} below), and so  \eqref{eq:qToNq} follows if $|L(1,\psi)| \gg \log q$.

Conversely, by \eqref{eq:GS2b} of Proposition \ref{prop:lambdaOpt} ((which is a consequence of Proposition \ref{prop: logsum}), we have
\[
\mathcal L:= \bigg|\sum_{n \leq N_q} \frac{\chi\bar{\xi}(n)}{n}\bigg| \ll (\log N_q) \exp( -\{ \lambda+o(1)\} \mathbb D(\psi,1;N_q)^2 )
\]
so that $\exp(-\mb{D}(\psi,1;N_q)^2)\geq (\frac{ \log N_q }{\mathcal L})^{-1/\lambda+o(1)}\gg 
(\frac{\mathcal L}{ \log N_q })^2$.
Now as $\psi$ is non-exceptional we have (see section \ref{sec: shortSums}) that
\begin{align*}
|L(1,\psi)|& \asymp \log q \, e^{ -\mb{D}( \psi ,1 ;q )^2 } =\log q \, e^{ -\mb{D}( \psi ,1 ;N_q )^2 -\mb{D}( \psi ,1 ;N_q,q )^2}\\ 
&\gg \log q \,  \bigg(\frac {\mathcal L}{ \log N_q }\bigg)^{2}   \bigg(\frac{\log N_q}{\log q}\bigg)^2=\frac{\mathcal L^2}{\log q}
\end{align*}
since $\mb{D}( \psi ,1 ;N_q,q )^2\leq 2\sum_{N_q<p\leq q} \frac 1p\leq 2\log (\frac{\log q}{\log N_q})+O(1)$.
If $\mathcal L\gg \log q$ then this establishes \eqref{eq:qToNq}; if $\mathcal L\gg (\log q)^\tau$ then this gives
$|L(1,\psi)|\gg (\log q)^{2\tau-1}$ showing that the example given after \eqref{eq:qToNq} is, in some sense, ``best possible''.
\end{proof}

\begin{proof} [Proof of Corollary \ref{cor: MandL}]
Suppose that $M(\chi) \gg \sqrt{q}\log q$. 
Proposition \ref{prop: BoundsOnMchi} shows that there is $\xi$ primitive of conductor $\ell \leq  \log q$ such that $\xi(-1) = -\chi(-1)$ and
$$
\log q \ll \frac{ M(\chi )}{ \sqrt{q} } \ll \frac{ \sqrt{\ell} }{ \phi(\ell) } \bigg| \sum_{n \leq N_q} \frac{ (\chi \bar{\xi})(n) }n \bigg|.
$$
The right-hand sum is $\ll \log q$, so $\ell \ll 1$. 
If $\chi\bar{\xi}$ is induced by a primitive character $\psi \pmod{ \ell^{\ast} }$ with $\ell^\ast | \ell$ then by Lemma 4.4 of \cite{GSPret}, 
\begin{equation} \label{eq: passtoPrim}
\sum_{n \leq N_q} \frac{ \psi(n) }n 
= \prod_{ p | \tfrac \ell{\ell^\ast} } \bigg( 1 - \frac{ \psi(p) }p \bigg)^{-1} \sum_{ n \leq N_q } \frac{ (\chi\bar{\xi})(n) }n + O(1),
\end{equation}
so $| \sum_{n \leq N_q} \tfrac { \psi(n) }n | \gg \log q$. By \eqref{eq:qToNq}, we obtain $|L(1, \psi )| \gg \log q$. 

Conversely, suppose there is a primitive character $\xi$ of conductor $\ell \ll 1$ with $\xi(-1) = -\chi(-1)$ such that 
$|L(1, \psi )| \gg \log q$, where $\psi$ is the primitive character that induces $\chi\bar{\xi}$.
By \eqref{eq:qToNq} and \eqref{eq: passtoPrim} we have $|\sum_{n \leq N_q} \tfrac{ (\chi\bar{\xi})(n) }n| \gg \frac{\phi(\ell)}{\sqrt{\ell}} \log q$,
and so Proposition \ref{prop: BoundsOnMchi} implies that
\[
M(\chi) \gg \frac{ \sqrt{ q \ell } }{ \phi(\ell) } \bigg|\sum_{n \leq N_q} \frac{(\chi\bar{\xi})(n)}{n}\bigg| \gg \sqrt{q}\log q,
\]
as claimed.
\end{proof}

\subsection{A generalization of Hal\'{a}sz's theorem}
Given a multiplicative function $f: \mb{N} \ra \mb{C}$, let $F(s) := \sum_{n \geq 1} f(n)/n^s$ denote its Dirichlet series, assumed to be analytic and non-zero for $\text{Re}(s) > 1$.
For such $s$ we write $-\tfrac{F'}{F}(s) = \sum_{n \geq 1} \Lambda_f(n)/n^s$. 

For fixed $\kappa\geq 1$ we restrict attention to those multiplicative functions $f$ for which $|\Lambda_f(n)|\leq \kappa \Lambda(n)$, where $\Lambda$ is the von Mangoldt function.
A generalization of Hal\'asz's Theorem to such $f$ (Theorem 1.1 of \cite{GHS2}) states that 
\begin{equation} \label{eq: genHal}
\sum_{n \leq x} f(n)  \ll  (1+M)e^{-M} x (\log x)^{\kappa-1} + \frac x{\log x} (\log\log x)^{\kappa},
\end{equation}
where $M$ is defined by
\[
e^{-M}(\log x)^\kappa =  \max\big\{ | s^{-1} F(s)| :\  s=1+\tfrac 1{\log x}+it \text{ with } |t|\leq (\log x)^\kappa\big\}.
 \]
In section \ref{sec: uppBdLog} we will apply this result (with $\kappa = 2$) to the convolution $1 \ast f$, where $f: \mb{N} \ra \mb{U}$ is multiplicative, in order to prove Proposition \ref{prop: logsum}.


\section{Large short character sums and large $L(1+it,\chi)$ values} \label{sec: shortSums}

In this section we will prove Proposition \ref{prop: LandB}.

\subsection{Truncations of  $L(1,\chi)$}
Let $t \in \mb{R}$. By partial summation we have 
\[
\bigg|\sum_{n>N} \frac{\chi(n)}{n^{1+it}}\bigg| \leq \frac{|S(\chi,N)|}N + |1+it| \int_N^\infty \frac{|S(\chi,y)|}{y^2} dy \leq (2+|t|) \frac{M(\chi)}N,
\]
so by the P\'{o}lya-Vinogradov theorem,
\begin{equation}\label{eq:L1Trunc}
L(1+it,\chi) = \sum_{n\leq N} \frac{\chi(n)}{n^{1+it}}+O\bigg( (2+|t|) \frac{\sqrt{q}\log q} N \bigg).
\end{equation}
We wish to also truncate the Euler product for $L(1+it,\chi)$ at $q$ when $|t|\ll 1$, losing at most a constant multiple. The prime number theorem in arithmetic progressions tells us that there exist constants $A,c>0$ such that if $L(s,\chi)$ has no exceptional zero
then
\begin{equation} \label{eq:PNTAP}
\psi(x,\chi) = \sum_{n \leq x} \chi(n)\Lambda(n) \ll x\exp\bigg( -c \frac{\log x}{\log q} \bigg) +\frac x{\log x}
\end{equation}
for all $x\geq q^A$. By partial summation we deduce that if $B>A$ then
\[
\sum_{p>q^B} \frac{\chi(p)}{p^{1+it}} =  - (1+it)\int_{q^B}^\infty \frac{\psi(u,\chi)}{u^{2+it}\log u}  du+O(1) \ll (1+|t|) \frac{e^{-cB}}B + 1.
\]
Now let $B=\max\{ 2A, (1/c) \log (1+|t|)\}$ so that since $\sum_{q<p\leq q^B}  \chi(p)/p^{1+it} \leq \sum_{q<p\leq q^B}1/p\ll \log B$, we have
\[
\sum_{p>q} \frac{\chi(p)}{p^{1+it}}  \ll \log B+ (1+|t|) \frac{e^{-cB}}B + 1 \ll 1+\log\log (1+|t|).
\]

Hence, if
$\chi$ is non-exceptional then for all $t \in \mb{R}$, $|t| \ll 1$,
\begin{align}\label{eq:L1Prod}
|L(1+it,\chi)| \asymp \bigg| \prod_{p\leq q} \bigg( 1 -\frac{\chi(p)}{p^{1+it}}\bigg)^{-1}\bigg| \asymp (\log q) e^{-\mathbb{D}(\chi,n^{it} ; q)^2}.
\end{align}

Taking $N=q$ in \eqref{eq:L1Trunc} and assuming $|t|\ll 1$, we have
\begin{align*}
L(1+it,\chi) = \sum_{n\leq q} \frac{\chi(n)}{n^{1+it}}+o(1)= ({1+it}) \int_1^q \frac{S(\chi,u)}{u^{2+it}} du +o(1)
\ll \int_1^q \frac{|S(\chi,u)|}{u^{2}} du +o(1)  ,
\end{align*}
and so, for any $c>0$, we have
\begin{equation} \label{eq: Lbd}
|L(1+it,\chi)|  \ll \bigg( c + \max_{q^c\leq x\leq q}\frac 1x |S(\chi,x)|\bigg) \log q  
\end{equation}
using the bound $|S(\chi,u)|\leq u$ for $u\leq q^c$. 

\subsection{Exceptional characters}
Landau proved that there exists an absolute constant $c > 0$ such that
for any $Q$ sufficiently large there is \emph{at most} one $q \leq Q$, one primitive real character $\chi \pmod{q}$ and one real number $\beta \in (0,1)$
such that 
\[
L(\beta,\chi)=0 \text{ and } \beta \geq 1-\frac{c}{\log Q}.
\]
If such a triple $(q,\chi,\beta)$ exists then we call
$q$ an \emph{exceptional modulus}, $\beta$ an \emph{exceptional zero} and $\chi$ an \emph{exceptional character}. 

If exceptional zeros exist there must be infinitely many of them (otherwise we can decrease $c$ as needed).
If $\{\beta_j\}_j$ is a sequence of exceptional zeros and $\{q_j\}_j$ is the corresponding set of exceptional moduli then
\[
(1-\beta_j)\log q_j \ra 0 \text{ as $j \ra \infty$.}
\]
It is an important open problem to obtain effective lower bounds for $1-\beta$.
Siegel's theorem (see e.g., \cite[Thm. 5.28]{IK}) states that if $\beta$ is the largest real zero of $L(s,\chi)$ then 
$1-\beta \gg_{\eps} q^{-\eps}$ for any $\eps > 0$, but the implicit constant is ineffective unless $\eps \geq 1/2$.

If $L(s,\chi)$ has an exceptional zero then  $\chi(p) = -1$ for many ``small'' primes. This suggests (but does not directly imply) the following result:

\begin{lemma}\label{lem: excepChi}
Suppose that $\chi$ is an exceptional character modulo $q$. Then: \\
a) $|L(1+it,\chi)| = o(\log q)$ when $|t|\ll 1$,  and  \\
b) for fixed $c > 0$ we have $|S(\chi,x)| = o_{q\to\infty}(x)$ for all $x \geq q^c$.
\end{lemma}


\begin{proof} 
a) ($t=0$):\ As $\chi$ is exceptional it must be real, and there is a $\beta \in ( 0 ,1)$ such that $L(\beta, \chi) = 0$ with $\eta := (1-\beta) \log q = o(1)$. 
By the truncation argument in \eqref{eq:L1Trunc},
\begin{align*}
L(\beta,\chi) &= \sum_{n \leq q} \frac{\chi(n)}{n}n^{1-\beta} + O\bigg(\frac{M(\chi)}{q^{\beta}}\bigg) = \sum_{n \leq q} \frac{\chi(n)}{n}\bigg(1 + O(\eta)\bigg) + O(q^{\tfrac 12-\beta} \log q) \\
&= \sum_{n \leq q} \frac{\chi(n)}{n} + O\bigg( \eta \log q \bigg)
\end{align*}
since $\eta \gg q^{-o(1) }$. By \eqref{eq:L1Trunc} we deduce that for any $\epsilon > 0$, 
\begin{align}\label{eq: SiegelL1Trunc}
|L(1,\chi)|= \bigg|\sum_{n \leq q} \frac{\chi(n)}{n}\bigg|+O(1/q^\epsilon)  \ll \eta  \log q
\end{align}
since $L(\beta,\chi) = 0$, which implies a).

 b) \  We use the above to observe that
$$
\frac{1}{q}\sum_{n \leq q} (1 \ast \chi)(n) = \frac{1}{q}\sum_{n \leq q} \chi(n) \llf \frac{q}{n}\rrf = \sum_{n \leq q} \frac{\chi(n)}{n} + O(1)  \ll \eta \log q + 1;
$$
on the other hand we have
$$
\frac{1}{q}\sum_{n \leq q} (1 \ast \chi)(n) \gg e^{-u e^{u/2}} \log q + O(1) \text{ where } u= \mb{D}(\chi,1;q)^2
$$ 
by \cite[(3.5)]{GSMinTrunc}, so that $\mb{D}(\chi,1;q)^2 = u \geq   \log\log( \tfrac 1\theta ) + O(1)$ where 
$\theta := \max\{ \eta, \tfrac 1{\log q} \}$.\footnote{One can obtain the better lower bound $\mb{D}(\chi,1;q)^2\geq \{ 2+o(1)\} \log\log( \tfrac 1\eta )$ by summing \cite[(3.23)]{TT} over all $m\ll \sqrt{\log 1/\eta}$ (note that their $\eta$ is our $1/\eta$).}
If  $x \in [q^c,q]$ that  $\mb{D}(\chi,1;x)^2 = \mb{D}(\chi,1;q)^2 +O(1)\geq \log\log( \tfrac 1\theta ) + O(1)$. 
 Therefore since $\chi$ is real, Hall and Tenenbaum's estimate \eqref{eq:HTbd} yields
\[
|S(\chi,x)| \ll x e^{-\tau \mb{D}(\chi,1;x)^2}   \ll \frac{x}{ (\log( 1/\theta ))^{\tau} } = o(x).  
\]

a) ($|t|\ll 1$):\ We insert the bound from (b) into \eqref{eq: Lbd}, and let $c\to 0$ to deduce our result.
\end{proof}

\begin{proof} [Proof of Proposition \ref{prop: LandB}]
We may assume that $\chi$ is an unexceptional character, since the result follows vacuously when 
$\chi$ is exceptional by Proposition \ref{prop: LandB}. 

Now if $|L(1+it,\chi)|\gg \log q$ for some $|t|\ll 1$ and if $c>0$ is sufficiently small  there exists  $x\in [q^c,q]$ for which
$|S(\chi,x)| \gg  x$ by \eqref{eq: Lbd}.

Now suppose that $|S(\chi,x)|\gg x$ for some  $x \in [q^{\kappa},q]$ for some fixed $\kappa \in (0,1]$.
For  a sufficiently large constant  $T$, Hal\'asz's Theorem  \eqref{eq: HalThm}   implies that   $\mathbb{D}(\chi,n^{it};x)\ll 1$ for some $|t|\leq T$, and so 
\begin{align}\label{eq:DchiTBd}
\mathbb{D}(\chi,n^{it} ; q)^2 = \mathbb{D}(\chi,n^{it};x)^2  + \sum_{x<p \leq q } \frac{1-\text{Re}(\chi(p)p^{-it})}{p} \ll 1 + 2\log\bigg( \frac{\log q}{\log x} \bigg) \ll_{\kappa} 1.
\end{align}
Then \eqref{eq:L1Prod} implies that $|L(1+it,\chi)|\gg_\kappa \log q$ as $\chi$ is non-exceptional.

Suppose now that $\chi^k=\chi_0$ with $k \ll 1$. As $x > q^{\kappa}$ we note then that
$$
x \ll |S(\chi,x)|\leq \#\{ n\leq x: (n,q)=1\} \ll \frac{\phi(q)}q x,
$$
which implies that $\sum_{p|q} \tfrac 1p \ll 1$. 
Next, repeatedly using the triangle inequality \eqref{eq: triIneq} for $\mathbb{D}$ together with \eqref{eq:DchiTBd},
\[
\mathbb{D}(1,n^{ikt} ; q)=\mathbb{D}(\chi^k,n^{ikt} ; q)+O\bigg( 1+  \sum_{p|q} \frac 1p \bigg) \leq k\, \mathbb{D}(\chi,n^{it} ; q)+O(1)\ll_{\kappa} 1.
\]
By \eqref{eq: Dfornit} we deduce that $\log(1+k|t|\log q)\ll_{\kappa} 1$, so that $|t| \ll_{\kappa} \tfrac 1{\log q}$. It follows then that $\mathbb{D}(\chi,1;q)^2=\mathbb{D}(\chi,n^{it},q)^2+O_{\kappa} (1)\ll 1$, and so  
$|L(1,\chi)| \asymp \log q \, e^{-\mb{D}(\chi,1;q)^2} \gg_{\kappa} \log q$ by \eqref{eq:L1Prod}.
\end{proof}

We would like to deduce that $|L(1,\chi)|\gg \log q$ from $|L(1+it,\chi)|\gg \log q$ with $|t|\ll 1$,  for characters of higher order. This is not necessarily guaranteed, though we can prove the following.

\begin{lemma} Let $\chi$ be a complex character mod $q$ and fix $T\geq 1$.  
If $|L(1+it,\chi)|\gg \log q $ with $|t|\leq T$ then
$|L(1+it_0,\chi)|\gg \log q $ where $t_0 = t(\chi; q, T)$ and 
$ |t-t_0|\ll \tfrac 1{\log q}$.
\end{lemma}

\begin{proof} Since $\chi$ is not real, it is non-exceptional. 
By \eqref{eq:L1Prod} we see that
\[
|L(1+it,\chi)| \gg \log q \text{ if and only if } \mb{D}(\chi,n^{it};q) \ll 1.
\]
Let $t_0 = t(\chi; q, T)$ so that $\mb{D}(\chi , n^{it_0} ;q) \leq \mb{D}(\chi , n^{it} ;q)\ll 1$, and therefore
$|L(1+it_0,\chi)| \gg \log q$ by \eqref{eq:L1Prod}. Moreover
\[
\mb{D}(1,n^{i(t-t_0)};q) =\mb{D}(n^{it_0},n^{it};q) \leq \mb{D}(\chi,n^{it};q) + \mb{D}(\chi , n^{it_0} ;q) \ll 1
\]
by \eqref{eq: triIneq},
and so we deduce that $|t-t_0| \ll \tfrac 1{\log q}$ by \eqref{eq: Dfornit}.
\end{proof}


\section{A variant of Hal\'asz's Theorem} \label{sec: uppBdLog}
In this section we will prove Propositions \ref{prop: logsum} and \ref{prop:lambdaOpt}(a), our various upper bounds for logarithmic averages.

\begin{proof} [Proof of Proposition \ref{prop: logsum}] As in the proof of Lemma \ref{lem: excepChi},
\[
\sum_{m\leq x} (1*f)(m)=\sum_{n \leq x} f(n) \left\lfloor \frac{x}{n} \right\rfloor = x \sum_{n \leq x} \frac{f(n)}{n} +O(x).
\]
Applying \eqref{eq: genHal}   to the mean value of $1*f$  with $\kappa=2$, we then obtain
\[
\sum_{n \leq x} \frac{f(n)}{n}  =\frac 1x\sum_{m\leq x} (1*f)(m)+O(1)
\ll (1+M) e^{ -M } \log x +1 \
\]
where $e^{-M}(\log x)^2=| \tfrac{1}s\zeta(s)F(s)| $ with $s = 1 + 1/\log x + it$, for some real $t, |t|\leq (\log x)^2$.
We have $|\zeta(s)| \leq\zeta(1 + \frac 1{\log x})\ll \log x$,
\begin{equation} \label{eq:BdFs}
|F(s)| \asymp \zeta(1+\tfrac 1{\log x})\exp(-\mathbb D(f,n^{it};x)^2) \asymp (\log x)\exp(-\mathbb D(f,n^{it};x)^2)\leq \log x;
\end{equation}
and $M\ll \log\log x$.
If $|t|\geq 1$ then 
$|\zeta(s)|\ll \log (2+|t|)$, and so
\[
\bigg|\sum_{n \leq x} \frac{f(n)}{n}\bigg| \ll (1+M) \frac{|F(s)|}{\log x} \frac{\log(2+|t|)}{1+|t|} \ll \log\log x.
\]
We henceforth assume that $|t|\leq 1$. We have 
\[
|\zeta(s)/s| \asymp \log y_t \text{ where } y_t=\min\{ e^{1/|t|},x\},
\]
and  
\[
e^{-M} = \frac{|\zeta( s) F( s)/s|}{(\log x)^2} \asymp \exp\bigg( - \min_{|\tau| \leq 1} (\mb{D}(1,n^{i\tau};x)^2 + \mb{D}(f,n^{i\tau};x)^2 ) \bigg).
\]
By \eqref{eq: Dfornit}, $\mb{D}(1,n^{it};x)^2 = \log(\frac{\log x}{\log y_t}) + O(1)$, so we deduce that
\begin{align}
\sum_{n \leq x} \frac{f(n)}{n} &\ll \log y_t \, \exp(- \mb{D}(f,n^{it};x)^2 ) \cdot \mathcal L \label{eq:ytBd} \\
&\asymp \mc{L} \, e^{ -\mc{L} } \log x, \label{eq:mcLBd}
\end{align}
where we have set
\[
\mathcal L := 1 + \mb{D}(f,n^{it};x)^2 + \log(\tfrac {\log x}{\log y_t}) \ll \log\log x.
\]  
From \eqref{eq:ytBd} and $\mb{D}(f,n^{it};x)^2 \geq M(f; x,1)$ we obtain \eqref{eq:GS.2} when $|t| \leq \tfrac 1{\log x}$ and \eqref{eq:GS2} when $\tfrac 1{\log x} < |t| \leq 1$.  
\end{proof}

The next proof continues on using the results in the previous proof:

 \begin{proof} [Proof of Proposition \ref{prop:lambdaOpt}(a)] If $|t| \leq \tfrac 1{\log x}$ then $\mb{D}(f , 1 ; x) = \mb{D}(f, n^{it} ; x) + O(1)$, and in this case we also obtain \eqref{eq:GS2b} (for the previous proof) with the better constant $1$ in place of $\lambda$.
 
 For $\tfrac 1{\log x}\leq |t|\leq 1$, we now prove the lower bound   $\mc{L} \geq \lambda \, \mb{D}(f,1;x)^2+O(1)$. When we  substitute this into \eqref{eq:mcLBd}, we obtain \eqref{eq:GS2b}
 since $y \mapsto ye^{-y}$ is a decreasing function for $y \geq 0$. 

First, as $p^{-it}=1+O(|t|\log p)$ when $p \leq y_t$, we obtain
  \[ 
\mathbb D(f,n^{it};y_t)^2-\mathbb D(f,1;y_t)^2 =
\sum_{p\leq y_t} \frac{  \text{Re} (f(p)(1-p^{-it} ) )}p\ll  |t| \sum_{p\leq y_t}  \frac{\log p}p 
\ll 1.
\]

The prime number theorem implies that 
\begin{align}\label{eq:optimize}
\sum_{y_t<p\leq x} \frac{ \text{Re} (f(p)(p^{-it} -\lambda) )}p
 \leq \sum_{y_t<p\leq x} \frac{ |p^{-it} -\lambda|}p  
&= \bigg(\int_0^1 |e(u) - \lambda| du \bigg) \log\bigg(\frac{\log x}{\log y_t}\bigg) + O(1) \nonumber\\
&=  (2-\lambda) \sum_{y_t<p\leq x} \frac{ 1 }p+O(1) ,
\end{align}
using the definition of $\lambda$. Re-organised, this gives
\[
 \mathbb D(f,n^{it};y_t,x)^2+\log\bigg(\frac{\log x}{\log y_t}\bigg) + O(1)  \geq \lambda \, \mathbb D(f,1;y_t,x)^2
+ O(1),
\]
so that 
\begin{align*}
\mc{L} &= (\mb{D}(f,n^{it};y_t,x)^2 + \log(\tfrac {\log x}{\log y_t})) + \mb{D}(f,n^{it};y_t)^2 +1\\
& \geq \lambda \, \mb{D}(f,1;y_t,x)^2 + \mathbb{D}(f,1;y_t)^2 +O(1) \geq \lambda \, \mb{D}(f,1;x)^2 + O(1). \qedhere
\end{align*}
\end{proof}

\section{Large character sums}

\subsection{Consequences of repulsion}
Suppose that we are given a primitive character $\chi \pmod q$.
Fix $A>0$.   For each  primitive character $\psi \pmod \ell$ with   $\ell\leq (\log q)^A$ select $|t| \le 1$ for which 
$\mathbb{D}(\chi ,\psi \, n^{it};q)$ is minimized. Index the pairs $(\psi, t)$ so
that  $(\psi_j, t_j)$ is the pair that gives the $j$-th smallest
distance $\mathbb{ D}(\chi ,\psi_j \, n^{it_j};q)$ (breaking ties arbitrarily if needed).  
A simple modification of \cite[Lem. 3.1]{BaGrSo} shows that for each $k\geq 2$ we
have
\[
 \mathbb{ D}(\chi ,\psi_k \, n^{it_k};q)^2 \ge (c_k+o(1)) \log \log q 
\]
where $c_k\geq 1-\tfrac 1{\sqrt{k}}$. As any $1 \leq n \leq N \leq q$ has $P(n) \leq q$, \cite[Thm. 6.4]{LamMan} yields
\begin{equation} \label{eq: logChiUppBd}
\sum_{n\leq N} \frac{(\chi\bar{\psi_k})(n)}n 
= \sum_{ \substack{ n \leq N \\ P(n) \leq q} } \frac{(\chi\bar{\psi_k})(n)}n 
\ll (\log q) e^{-\mb{D}(\chi,\psi_k \, n^{it_k}; q)^2} + 1 
\ll (\log q)^{1-c_k + o(1)}.
\end{equation}
Under the additional hypothesis that $\psi_1\psi_2$ is an even character (which follows if we restrict attention to $\psi$ with $(\chi\psi)(-1)=-1$) we can take any
$c_2 > 1-\tfrac 2\pi$ as we will see below in Lemma \ref{lem: 2/Pi}.

 For each integer $k\geq 1$ we define
 \[
 \gamma_k := \frac{1}{k} \sum_{a=0}^{k-1} | \cos(\pi a/k)| = \frac{1}{k} \sum_{a=0}^{k-1} \big| 1+e\big( a/k \big)\big| .
\]
Using
the Fourier expansion 
 \begin{align} \label{eq_fourier}
|1+e(\alpha)| = \frac{4}{\pi}\bigg(1- \sum_{d \neq 0} \frac{(-1)^{d}}{4d^2-1} e(d\alpha)\bigg), \quad \alpha \in \mb{R}/\mb{Z},
\end{align}
we easily deduce that
 \begin{equation} \label{eq: ckdefn}
 \gamma_k= \frac 2\pi \Bigg(  1   - 2\sum_{\substack{r\geq 1 \\ k|r}} \frac{(-1)^r} {4r^2-1} \Bigg) = \begin{cases} 
 \frac {\text{\rm cosec}(\pi/2k)}{k} &\text{if} \ k \ \text{is odd}, \\
 \frac   {\cot(\pi/2k)}k &\text{if} \ k \ \text{is even};
 \end{cases}   
\end{equation}
(see also Lemma 5.2 of \cite{GHS}).
Therefore $\gamma_k= \tfrac 2\pi +O(\tfrac 1 {k^2})$, and the   $\gamma_k$, with $k$ even, increase towards $\tfrac 2\pi=0.6366\ldots$:
 \[
 \gamma_2=\tfrac 12< \gamma_4=0.6035\ldots < \gamma_6 = 0.6220\ldots < \gamma_8 =0.6284\ldots
 \]
 whereas the $\gamma_k$, with $k$ odd, decrease towards $\tfrac 2\pi$:
 \[
\gamma_1=1> \gamma_3=\tfrac 23> \gamma_5=0.6472\ldots > \gamma_7 = 0.6419\ldots > \gamma_9 =0.6398\ldots
 \]
For $k>1$, we have $\gamma_k\leq \tfrac 23$.

\begin{lemma} \label{lem: 2/Pi}
Fix $C > 0$ and let $m \leq (\log q)^{C}$. 
Let $t_1,t_2 \in \mb{R}$ be chosen such that $t := t_2-t_1$ satisfies $|t| \ll 1$. 
Let $\chi_1$ and $\chi_2=\chi_1\xi$ be characters with modulus in $[q,q^2]$, where $\xi \pmod m$ is an odd character. 
Suppose that 
 $\mathbb{ D}(\chi_2, n^{it_2}; q)  \geq \mathbb{ D}(\chi_1, n^{it_1}; q)$. Then for any $\eps > 0$,
 \[
 \mathbb{ D}(\chi_2, n^{it_2};q)^2   \geq ( 1 - \tfrac{2}{\pi} - \eps \theta )\log\log q  + O(\log\log\log q),
\]
where $\theta := 1$ if $\xi^j$ is exceptional for some $1 \leq j \leq \min\{ \phi(m) , \log\log q \}$, and $\theta := 0$ otherwise.
\end{lemma}

\begin{proof} Suppose that $\xi \pmod m$ has order $k$. 
We note that
\begin{align}
\mb{D}( \chi_2 , n^{it_2} ; q)^2 &\geq \frac{1}{2} \bigg( \mb{D}( \chi_2 , n^{it_2} ; q)^2 + \mb{D}( \chi_1 , n^{it_1} ; q)^2 \bigg)  \notag \\
&= \log\log q - \frac{1}{2}\text{Re}\bigg( \sum_{p \leq q} \frac{\chi_1(p)p^{-it_1} + \chi_2(p)p^{-it_2}}{p} \bigg) + O(1) \notag\\
&= \log\log q - \frac{1}{2}\text{Re}\bigg( \sum_{p \leq q} \frac{\chi_1(p)p^{-it_1} (1+ \xi(p)p^{-it}) }{p} \bigg) + O(1) \notag \\
&\geq \log\log q - \frac{1}{2}   \sum_{p \leq q} \frac{ |1+ \xi(p)p^{-it}| }{p}  + O(1)\label{eq:useful}
\end{align}
Using \eqref{eq_fourier} we deduce that 
\begin{align*}
\sum_{ p \leq q } \frac{ | 1 + \xi(p)p^{-it} | }{p} 
= \frac{4}{\pi} \bigg( \log\log q - \sum_{1 \leq |d| \leq D} \frac{ (-1)^{d} }{ 4d^2-1} S_d\bigg) + O(1),
\end{align*}
where  $D := \log\log q$ and
$$
S_d := \sum_{p \leq q} \frac{ \xi(p)^d }{ p^{1+idt} }.
$$
If $k$ divides $d$ then $\xi(p)^d = 1_{p\nmid m}$, and so 
\begin{align*}
S_d &= \sum_{p \leq q} \frac{1}{ p^{1+idt} } + O\bigg( \sum_{p|m} \frac{1}{p} \bigg) = \sum_{p \leq q} \frac{1}{ p^{1+idt} } + O(\log\log\log\log q).
\end{align*}
The sum over $p\leq z := \min\{q, e^{2\pi/|dt|}\}$ equals 
\begin{align*}
\sum_{p \leq z} \frac 1p + O\bigg(|dt| \sum_{p\leq z} \frac{\log p}{p}\bigg) &=  \log( \min\{ \log q ,  \tfrac { 2\pi }{ |dt| } \} ) +O(1) \\
&= \log( \min\{ \log q , \tfrac 1{|t|} \} ) + O( \log\log \log q ).
\end{align*}
The remaining set of primes $z < p \leq q$ is non-empty only if $z= e^{2\pi/|dt|}$, which happens when $|dt| > \frac{2\pi}{\log q}$. Let $\sigma := \tfrac{dt}{|dt|} \in \{-1,+1\}$.  Applying the prime number theorem,
\begin{align*}
\sum_{z<p \leq q} \frac{1}{ p^{1+idt} } =  \int_{e^{2\pi/|dt|}}^q u^{-idt} \frac{du}{u\log u} + O(1)=\int_1^{X} e(-\sg v) \frac{dv}{v} +O(1)=O(1),
\end{align*}
putting $v = \tfrac{|dt| \log u}{2\pi}$ and $X = \tfrac{ \log q}{\log z}$ ($\geq 1$) and then integrating by parts.
We deduce that if $k$ divides $d$ with $d\leq D$ then
$S_d = \log(\min\{ \log q, \tfrac 1{|t|} \} ) + O(\log \log\log q)$.

If $k$ does not divide $d$ let $\alpha:=\xi^d$, a non-principal character mod $m$, and let $T=dt$ so that $|T|\ll D$. 
For $\epsilon > 0$ small, let 
\[
Y := \begin{cases} \exp( (\log m)^{ 10 } + T ) &\text{ if $\alpha$ is non-exceptional,} \\
			\exp( ( \log q )^{ \epsilon }) &\text{ if $\alpha$ is exceptional.}
\end{cases}
\] 
Trivially bounding the primes $p \leq Y$ we deduce that 
\[
S_d = \sum_{Y<p \leq q} \frac{ \alpha(p) }{ p^{1+iT} }  + O(\log\log Y ).
\]
Since $\psi(y,\alpha) \ll \tfrac y{\log y}$ for $y \geq Y$ by the Siegel-Walfisz theorem when $\alpha$ is exceptional,
and using  \eqref{eq:PNTAP} otherwise,
this is
\begin{align*}
\int_Y^q \frac{d \psi(y,\alpha)}{y^{1+iT}\log y}+O\bigg( \frac 1Y \bigg) &= \left[\frac{ \psi(y,\alpha)}{y^{1+iT} \log y} \right]_Y^q+ (1+iT)\int_Y^q \frac{\psi(y,\alpha)}{y^{2+iT}\log y} dy + O\bigg( \frac 1{ (\log Y)^2 } \bigg)\\
&\ll \frac{ 1}{(\log Y)^2} + T \int_Y^q \frac{ 1}{y(\log y)^2} dy \ll 1.
\end{align*}
We conclude that $S_d \ll \log\log Y$ whenever $k \nmid d$.

Putting these estimates into \eqref{eq: ckdefn}, we find that
\begin{align*}
\sum_{p \leq q} \frac{ |1 + \xi(p)p^{-it}| }{p} &= \frac{4}{\pi} \log\log q + 2 \bigg(\gamma_k - \frac{2}{\pi} \bigg) \log( \min\{ \log q, \tfrac 1{|t|} \} )  \\
&+ O(\log\log\log q +   \log\log Y),
\end{align*}
and $\log\log Y\ll  \eps \theta \log\log q + \log\log\log q$.
Changing $\epsilon$ by a (possibly ineffective) constant factor if needed, we deduce from \eqref{eq:useful} that
\begin{align*}
\mathbb{ D}(\chi_2, n^{it_2}; q)^2   \geq (1 - \tfrac{2}{\pi} - \epsilon\theta) \log\log q +   (\tfrac{2}{\pi}-\gamma_k ) \log ( \min\{ \log q, \tfrac 1{|t|} \} )    + O(\log\log\log q).
\end{align*}
Now since $\xi$ is odd, its order $k$ must be even and therefore $\gamma_k\leq \tfrac 2\pi$;
 the result follows since the middle term is $\gg -O(1)$.
 \end{proof}


For any $\epsilon>0$ let $K>1/\epsilon^2$ so that $1-c_K<\epsilon$. 
As a consequence of \eqref{eq: logChiUppBd} and Lemma \ref{lem: 2/Pi}, established just below, we have 
\begin{equation} \label{eq: chsumbd}
\max_{1 \leq N \leq q} \bigg|\sum_{n\leq N} \frac{(\chi\bar{\psi})(n)}n \bigg| =
\begin{cases} O( (\log q)^{\epsilon} ) &\text{ if } \psi\ne \psi_k \text{ for all } k<K,\\
O_{\eps}( (\log q)^{\tfrac 2\pi+\eps} ) &\text{ if } \psi= \psi_k \text{ for some } k, 1<k <K.
\end{cases}
\end{equation}
The implicit constant in the second case is effective as long as $\psi^j$ is non-exceptional for all $j \geq 1$, and in this case $\tfrac 2\pi + \eps$ can be replaced by $\tfrac 2\pi +o(1)$.

\subsection{Working with large character sums}
In this subsection we set the stage for the proof of Proposition \ref{prop: BoundsOnMchi}. We recall that $\Delta \in (\tfrac 2 \pi, 1)$   and $q \geq 3$ are given, and
$$
R_q := \exp( \tfrac {(\log q)^\Delta }{\log\log q} ), \quad r_q := (\log q)^{2-2\Delta} (\log\log q)^4.  
$$

\begin{proposition} \label{prop:ChiLarge}
For any given primitive character $\chi \pmod q$ there exists a primitive character $\xi \pmod{\ell}$ with $\ell \leq r_q$ 
and $(\chi\xi)(-1)=-1$ such that if $|\alpha - \tfrac bm| \leq \frac{1}{m R_q}$ with $m \leq r_q$ and $N := \min\{ q, \frac{ 1}{ |m\alpha - b|} \}$ then 
\[
\sum_{1 \leq |n| \leq q} \frac{ \chi(n) e(n\alpha)}n = 1_{\ell | m}
\frac  {2\eta g( \xi  )}{\phi(m)} 
 \prod_{p|\tfrac m\ell} (1-(\bar{\chi} \xi)(p))
 \sum_{1 \leq n \leq N } \frac{ ( \chi\bar{\xi} )(n) }{n} +o((\log q)^{\Delta}),
\]
where, whenever $\ell | m$ we write $\ell_q$ to denote the largest divisor of $m/\ell$ that is coprime to $q$, $m_q \ell_q = m/\ell$ and 
$$
\eta :=\mu(m_q)  \chi(\ell_q)   \bar{\xi}(b)  \xi(m_q) \in S^1 \cup \{0\}.
$$ 
\end{proposition}

\begin{proof} 
By \eqref{eq: NewSum} we have
\begin{align} \label{eq: NewSum2}
\sum_{1 \leq |n| \leq q} \frac{ \chi(n) e(n\alpha)}{n} = \sum_{1 \leq |n| \leq N} \frac{ \chi(n) e(n \tfrac bm)}{n} + O(\log\log q).
\end{align}
With the intent of replacing exponentials by Dirichlet characters on the right-hand side, we split the $n$th summand according to the common factors of $n$ with $m$. Therefore if $(n,m)=d$ with $m=cd$ and $n=rd$ we have
\begin{align*}
 \sum_{ 1 \leq |n| \leq N } \frac{\chi(n) }{n} e(n\tfrac bm) &= \sum_{ cd = m } \frac{ \chi(d) }{d} \sum_{ 1 \leq |r| \leq N/d \atop (r,c) = 1 } \frac{ \chi(r) }{r} e(r\tfrac bc) \\
&= \sum_{cd = m} \frac{ \chi(d) }{ d \phi(c) } \sum_{ \psi \pmod{c} } \bar{\psi}( b ) g( \psi ) \sum_{ 1 \leq |n| \leq N/d } \frac{ \chi \bar{\psi}(n) }{n} \\
&= 2 \sum_{ cd = m } \frac{ \chi(d) }{ d \phi(c) } \sum_{ \psi \pmod{c} \atop \chi\psi(-1) = -1 } \bar{\psi}(b) g( \psi ) \sum_{1 \leq n \leq N/d} \frac{ \chi \bar{\psi}(n) }{n},
\end{align*}
since if $(k,c)=1$ then $e(\tfrac kc)=\frac{1 }{ \phi(c) } \sum_{ \psi \pmod{c} } \bar{\psi}( k ) g( \psi )$, and then noting
the $n$ and $-n$ terms cancel if $\chi\psi(-1) =1$.

To control the size of $g( \psi )$ we split the sum over characters modulo $c$ according to the primitive characters that induce them.
Since each $\psi$ factors as $\psi^{\ast}\psi_0^{ (f) }$, where $\psi^{\ast}$ is primitive modulo $e$ and $\psi_0^{(f)}$ is principal modulo $f$ with $ef = c$, the right-hand side of the above sum is
\begin{align}
&2 \sum_{efd = m} \frac{ \chi(d) }{ d \phi(ef) } \asum_{ \psi^{\ast}  \pmod{e} \atop \chi\psi^{\ast}(-1) = -1 } \bar{\psi}^{\ast}( b ) g( \psi^{\ast} \psi_0^{(f)} ) \sum_{ 1 \leq n \leq N/d \atop (n,f) = 1 } \frac{ \chi \bar{\psi}^{\ast}(n) }{n} \nonumber\\
&= 2\sum_{ efd = m \atop (e,f)=1 } \frac{ \mu(f)\chi(d) }{ d \phi(ef) } \asum_{ \psi^{\ast} \pmod{e} \atop \chi\psi^{\ast}(-1) = -1 } \psi^{\ast}( f\bar{b} ) g( \psi^{\ast} ) \sum_{1 \leq n \leq N/d \atop (n,f) = 1} \frac{ \chi \bar{\psi}^{\ast}(n) }{n}, \label{eq:decompCharSum}
\end{align}
since $g( \psi^{\ast} \psi_0^{(f)} )=\psi^{\ast}(f)\mu(f) g(\psi^{\ast})$.

Fix $ef|m$ and $\psi^{\ast} \pmod{e}$. We extend the inner sum in \eqref{eq:decompCharSum} to all $n\leq N$ as
\begin{align}\label{eq:removeTrunc}
\sum_{1 \leq n \leq N/d \atop (n,f) = 1} \frac{ \chi \bar{\psi}^{\ast}(n) }{n} = \sum_{1 \leq n \leq N \atop (n,f) = 1} \frac{ \chi \bar{\psi}^{\ast}(n) }{n}+ O(\log (2d)).
\end{align}
By Lemma 4.4 of \cite{GSPret}, 
\begin{align} \label{eq:removeGCD}
\sum_{1 \leq n \leq N \atop (n,f) = 1} \frac{ \chi \bar{\psi}^{\ast}(n) }{n} = \prod_{p|f} \bigg( 1-\frac{ \chi \bar{\psi}^{\ast}(p) }{p} \bigg) \sum_{1 \leq n \leq N} \frac{ \chi \bar{\psi}^{\ast}(n) }{n} + O\bigg( (\log\log ( 2 + f ))^2 \bigg).
\end{align}
We next observe the identity
\[
\sum_{ fd = m/e \atop (e,f)=1} \frac{ \mu(f) \chi(d) }{d \phi(ef)}   \psi^{\ast}(f)   \prod_{p|f} \bigg( 1-\frac{ \chi \bar{\psi}^{\ast}(p) }{p} \bigg) 
= \frac{ \chi(e_q)\psi^{\ast}(m_q) \mu(m_q)}{\phi(m)}  \prod_{p|\tfrac me} (1-(\bar{\chi} \psi^{\ast} )(p)),
\]
where now $e_q$ is the largest divisor of $m/e$ which is coprime to $q$, and $m_qe_q=m/e$. The main terms from \eqref{eq:removeGCD} thus contribute
\begin{align}\label{eq: mainTerms}
\frac 2{\phi(m)}  \sum_{e|m} \mu(m_q) \chi(e_q)   \asum_{ \psi^{\ast}  \pmod{e} \atop \chi\psi^{\ast} (-1) = -1 } g( \psi^{\ast} )
  \psi^{\ast}(m_q \bar{b} )\prod_{p|\tfrac me} (1-( \bar{\chi} \psi^{\ast} )(p))
 \sum_{1 \leq n \leq N } \frac{ \chi \bar{\psi}^{\ast}(n) }{n}
 \end{align}
in \eqref{eq:decompCharSum}. By noting that $| g( \psi^{\ast} ) | = \sqrt{e}=\sqrt{m/df}$ the contribution of the error terms from \eqref{eq:removeTrunc} and \eqref{eq:removeGCD} in \eqref{eq:decompCharSum} is bounded by
\begin{align} \label{eq: errorTerms}
\ll \sqrt{m} \sum_{ fd | m } \frac{ ( \log 2f )  (\log 2d) }{d^{3/2} \phi(f)f^{1/2}  }    \ll \sqrt{m} \leq \sqrt{r_q} = o( (\log q)^{\Delta} ).
\end{align}

We now apply \eqref{eq: chsumbd} with $N = \min\{q, \tfrac 1{|m\alpha-b|}\}$. Set $\xi := \psi_1$, whose contribution,
\[
 \frac{2 \mu(m_q) \chi(\ell_q) \xi( \bar{b} m_q ) g( \xi ) }{ \phi(m) } 
\prod_{p| \tfrac m\ell } ( 1 - ( \bar{\chi} \xi )(p) )  \sum_{1 \leq n \leq N} \frac{ \chi \bar{\xi}(n) }{n}
\]
only appears in \eqref{eq: mainTerms} if the conductor $\ell$ of $\xi$ divides $m$. By \eqref{eq: chsumbd} the contributions to \eqref{eq: mainTerms} from all the characters $\psi\ne \psi_k$ for all $k < K$ is, for $\epsilon$ sufficiently small,
 \[
\ll (\log q)^{\epsilon}  \sum_{ e|m }  \frac {\sqrt{e}\, \tau(m/e)\phi(e)}{\phi(m)}   
\ll \sqrt{ m } (\log q)^{2\epsilon}\leq \sqrt{ r_q } (\log q)^{2\epsilon} \ll (\log q)^{1-\Delta+3\epsilon} =  o( (\log q)^{\Delta} ).
\]
Since the coefficient in front of each individual sum over $n$ in \eqref{eq: mainTerms} is bounded, again by \eqref{eq: chsumbd}
the contribution of the main terms from all of the characters $  \psi_k$ with $1<k<K$ is
$\ll_{\eps} K \cdot (\log q)^{\tfrac 2\pi+\eps} = o( (\log q)^{\Delta} )$, if $\eps$ is sufficiently small.
We insert these estimates into \eqref{eq: NewSum2} to obtain the result.
\end{proof}

The proof of Proposition \ref{prop:ChiLarge} is very similar to the proof of the main results in \cite{BrGr}.

\begin{proof} [Proof of Proposition \ref{prop: BoundsOnMchi}]
Let $\alpha \in [0,1)$ be chosen so that $M(\chi) = |S(\chi,\alpha q)|$. Applying \eqref{eq_PFE}, we have
\[
\frac{M(\chi)}{\sqrt{q}} = \frac{1}{ 2\pi} \bigg| \sum_{1 \leq |n| \leq q} \frac{  \chi(n) }{n} 
- \sum_{1 \leq |n| \leq q} \frac{ \chi(n) e(n\alpha ) }{n} \bigg| + O(1).
\]
Let $\xi := \psi_1$ once again, and let $\ell$ be its conductor. 
The proof is split up according to whether $\ell > 1$ or $\ell = 1$.

\noindent \textbf{Case 1:} Assume $\ell > 1$, so $\xi$ is non-trivial.
Suppose first that $| M(\chi) | \gg \sqrt{q} (\log q)^{\Delta}$. 
In light of \eqref{eq: minorArc}, $\alpha$ is on a major arc, so there is $\tfrac bm$ such that $m \leq r_q$ and $|\alpha-\tfrac bm| \leq \tfrac 1{mR_q}$, with $\ell | m$ by Proposition \ref{prop:ChiLarge}.
Note that if we vary $\alpha$ in the interval $\left[\tfrac bm - \tfrac 1{mR_q}, \tfrac bm + \tfrac 1{mR_q} \right]$ then $N = N(\alpha) = \min\{ q, \tfrac 1{|m\alpha-b|} \}$ varies in the range $R_q\leq N\leq q$. 
As $\ell > 1$, Proposition \ref{prop:ChiLarge} also shows that 
$\sum_{1 \leq |n| \leq q} \frac{ \chi(n) }n = o( (\log q)^{\Delta} )$,
and moreover, writing $m_q \ell_q = m/\ell$ as before,
\[
\frac{M(\chi)}{\sqrt{q}} = \frac{ \sqrt{\ell} }{ \pi \phi(m) }  
 \prod_{p|\tfrac m\ell} |1-( \bar{\chi} \xi)(p)| \cdot \bigg|
 \sum_{1 \leq n \leq N } \frac{ (\chi \bar{\xi})(n) }{n}  \bigg|
+o((\log q)^{\Delta})
\]
provided $\mu(m_q) \chi(\ell_q)  \xi(m_q)\ne 0$.
Evidently, to maximize the right-hand side we must have $N(\alpha)=N_q$.  

Next, we find $m = d\ell$, given $\xi$ and $N_q$, that maximizes 
\[
s_d:=\frac  1{\phi(d\ell)}  \prod_{p | d} | 1-(\bar{\chi} \xi)(p) |.
\] 
Suppose that $p^e\|d$ and $D=dp^e$. If $p|\ell$ then $s_{d}=s_{D}/p^e<s_D$ so we may assume that $p\nmid \ell$.
In that case $s_{d}\leq 2s_{D}/\phi(p^e)\leq s_D$ unless $p^e=2$. Hence  $d=1$ or $2$ and $\phi(d\ell)=\phi(\ell)$, and so 
\[
\frac{M(\chi)}{\sqrt{q}} = \frac{ \sqrt{ \ell } }{ \pi \phi(\ell) } \max\{ 1 , | 1 - ( \bar{\chi} \xi )(2) | \} \bigg| \sum_{ 1 \leq n \leq N_q } \frac{ (\chi \bar{\xi})(n) }{n} \bigg| 
+ o( (\log q)^{\Delta} ),
\]
which proves \eqref{eq: MchiAsymp} when $\ell>1$, and also that 
$| \sum_{n \leq N_q} \frac{ (\chi \bar{\xi})(n) }{n}  | \gg \frac{ \phi(\ell) }{ \sqrt{ \ell } } (\log q)^{ \Delta }$.

Conversely, assume that 
$\bigg| \sum_{n \leq N_q} \frac{ (\chi \bar{\psi})(n) }{n} \bigg| \gg \frac{ \phi(r) }{ \sqrt{r} } (\log q)^{ \Delta }$ 
for some primitive character $\psi$ of conductor $r \leq r_q$ with $\psi(-1) = -\chi(-1)$. 
In view of \eqref{eq: chsumbd}, it follows that $\psi = \xi$ and $r = \ell$. 
The assumption also implies that $\log N_q + O(1) \geq (\log q)^{\Delta}$, so $N_q \geq R_q$.

Selecting $\beta \in [ \tfrac 1\ell - \tfrac 1{\ell R_q} , \tfrac 1\ell + \tfrac 1{\ell R_q}]$ so that $N(\beta) = N_q \in [R_q,q]$ and 
applying Proposition \ref{prop:ChiLarge}, 
\begin{align} \label{eq: ellEqualsm} 
\bigg|\sum_{1 \leq |n| \leq q} \frac{\chi(n)e(n\beta)}{n}\bigg| 
=  \frac{2\sqrt{\ell}}{\phi(\ell)} \bigg|\sum_{1 \leq n \leq N_q} \frac{\chi\bar{\xi}(n)}{n}\bigg| + o( (\log q)^{\Delta} ) 
\gg (\log q)^{\Delta}.
\end{align}
Combining \eqref{eq_PFE} with \eqref{eq: ellEqualsm} and a second application of Proposition \ref{prop:ChiLarge}, we get
$$
M(\chi) \geq |S(\chi, \beta q)| 
\geq \frac{\sqrt{q}}{2\pi} \bigg|\sum_{1 \leq |n| \leq q} \frac{\chi(n)e(n\beta)}{n}\bigg| - \frac{ \sqrt{q} }{ 2\pi } \bigg| \sum_{1 \leq |n| \leq q} \frac{\chi(n)}{n} \bigg| + O( \log q )
\gg \sqrt{q} (\log q)^{\Delta},
$$
as required. 

\noindent \textbf{Case 2:} Assume now that $\ell = 1$ and $\xi$ is trivial so that $\chi$ is odd.
If $\alpha$ is on a minor arc then from \eqref{eq_PFE} and \eqref{eq: minorArc} we get 
\begin{align} \label{eq: minor}
\frac{ M(\chi) }{ \sqrt{q} } = \frac { 1- \chi(-1) }{ 2\pi }\bigg| \sum_{n \leq q} \frac{ \chi(n) }{n} \bigg| + o( (\log q)^{\Delta} ).
\end{align}
On the other hand, if $\alpha$ is on a major arc then by Proposition \ref{prop:ChiLarge},
\[
\frac{M(\chi)}{\sqrt{q}} =  \frac{ 1- \chi(-1) }{ 2\pi } \max_{\substack{ R_q\leq N\leq q \\ 1 \leq m \leq r_q}} \bigg| \sum_{n \leq q} \frac{ \chi(n) }{n} -    \frac{ \chi(\ell_q) \mu(m_q)}{\phi(m)}  \prod_{p|m} (1-\bar{\chi} (p)) \sum_{1 \leq n \leq N} \frac{ \chi(n) }{n} \bigg| + o( (\log q)^{\Delta} ).
\]
The coefficient of the sum up to $N$ is $\leq 2$ (which is attained if $m=2$ and $\chi(2)=-1$), so that by the triangle inequality
\begin{align} \label{eq: ell1UppBd}
\frac{M(\chi)}{\sqrt{q}} \leq \frac 3\pi \bigg| \sum_{n\leq N_q} \frac{\chi(n)}n \bigg| + o( (\log q)^{\Delta} ).
\end{align}
We obtain this as an equality when $\chi(2)=-1$ with $N_q=q$ and $m=2$.

By \eqref{eq: minor} and then by taking $m=1$ and $N=N_q$ above, we  obtain 
\begin{align} \label{eq: ell1LowBd}
\frac{M(\chi)}{\sqrt{q}} &\geq \frac 1\pi \max \bigg\{ \bigg| \sum_{n \leq q} \frac{ \chi(n) }{n} \bigg| , \bigg| \sum_{n \leq q} \frac{ \chi(n) }{n} -  \sum_{ n \leq N_q} \frac{ \chi(n) }{n} \bigg|  \bigg\} + o( (\log q)^{\Delta} ) \nonumber\\
&\geq \frac 1{2\pi} \bigg| \sum_{n\leq N_q} \frac{ \chi(n)}n \bigg| + o( (\log q)^{\Delta} ).
\end{align}
 \eqref{eq: ell1UppBd} and \eqref{eq: ell1LowBd} imply \eqref{eq: MchiAsymp}. Together these bounds also yield that $M(\chi) \gg \sqrt{q}(\log q)^{\Delta}$ if and only if $|\sum_{n \leq N_q} \tfrac{ \chi(n) }n| \gg (\log q)^{\Delta}$.
\end{proof}

 \begin{corollary}\label{cor: allLarge}
Let $\chi$ be a primitive character modulo $q$. Assume $M(\chi) \geq c_1 \sqrt{q}\log q$ and $|S(\chi,N)| \gg N$ for some $N \in [q^{c_2},q]$, with $c_1,c_2 \gg 1$. Then $|L(1,\chi)| \gg \log q$ , $\mb{D}(\chi,1;q) \ll 1$ and $\chi$ is odd.
\end{corollary}

\begin{proof}
Assume that $M(\chi) \geq c_1 \sqrt{q} \log q$ and $|S(\chi,N)| \gg N$ for some $N \in [q^{c_2},q]$. 
By Corollary \ref{cor: MandL} and Proposition \ref{prop: LandB}, there is $|t| \ll 1$ and $\ell \ll 1$ such that, simultaneously, 
\[
|L(1+it,\chi)| \gg \log q \text{ and } |L(1,\chi\bar{\xi})| \gg |L(1,\psi)| \gg \log q,
\]
where $\xi$ is primitive modulo $\ell$ and $\xi(-1) = -\chi(-1)$, and $\psi$ is the primitive character that induces $\chi\bar{\xi}$.
By \eqref{eq:L1Trunc},
\[
\bigg| \sum_{n \leq q} \frac { \chi(n) n^{-it} }{n} \bigg|, \bigg| \sum_{n \leq q} \frac{ \chi(n) \bar{\xi}(n) }{n} \bigg| \gg \log q.
\] 
Applying \eqref{eq:GS2b} of Proposition \ref{prop: logsum} we obtain $\mb{D}( \chi , n^{it} ; q) , \mb{D}( \chi , \xi ; q) \ll 1$,
so by \eqref{eq: triIneq}, 
\[
\mb{D}( \xi , n^{it} ; q) \leq \mb{D}( \chi , \xi ; q) + \mb{D}( \chi , n^{it} ; q) \ll 1.
\]
Therefore $\ell = 1$, else we let $Y := \exp( (\log (2\ell) )^{10} + |t| ) \ll 1$, and apply \eqref{eq:PNTAP} and partial summation as in the proof of Lemma \ref{lem: 2/Pi} to get
\[
\sum_{p \leq q} \frac{ \xi(p) p^{-it}}{p} = \sum_{Y < p \leq q} \frac{\xi(p)p^{-it}}{p} + O(\log\log Y) \ll 1,
\]
which implies that $\mb{D}(\xi, n^{it} ; q)^2 = \log\log q + O(1)$,   a contradiction. 

We deduce that $\xi$ is trivial so that $\chi(-1) =\chi\xi(-1) = -1$ and that  $|L(1,\chi)| \gg \log q$.
By Lemma \ref{lem: excepChi}, $\chi$ must be non-exceptional, so \eqref{eq:L1Prod} gives
$|L(1,\chi)| \asymp \log q \, e^{-\mb{D}(\chi,1;q)^2}$ and we deduce that $\mb{D}(\chi,1;q) \ll 1$. 
\end{proof}
 
%

\section{A class of examples}

\subsection{The set-up}
Let $g: \mb{R} \ra \mb{U}$ be a $1$-periodic function with $g(0) = 1$ and  Fourier expansion
\[
g(t) =  \sum_{n \in \mb{Z}} g_ne(nt)
\]
so that 
\[
g_n := \hat{g}(n):= \int_0^1 g(u) e(-nu) du \text{ for all integers } n,
\]   
and therefore $|g_n| \leq \int_0^1 |g(u)| du \leq 1$ for all $n$.
We will assume that $|g_n| \ll |n|^{-3}$ for all integers $n \neq 0$ (so that $\{g_n\}_n$ is absolutely summable).\footnote{The proof works provided  $|g_n| \ll 1/ |n|^{2+\epsilon}$ for all integers $n\ne 0$.}

Write $\gamma_0=g_0+1$ and $\gamma_n=g_n$ for all integers $n\ne 0$.
Then $\sum_{n \in \mb{Z}} \text{Re}(\gamma_n) = \sum_{n \in \mb{Z}} \text{Re}(g_n) +1=  \text{Re}(g(0))+1=2$ 
so that $\mu:=\max_n \text{Re}(\gamma_n)>0$.
Let $\mathcal L:=\{ \ell\in  \mb{Z}: \text{Re}(\gamma_\ell)=\mu\}$, which is a non-empty set, and finite as $|g_n| \ll |n|^{-3}$. Moreover
there exists $\delta>0$ such that $\text{Re}(g_n) \leq \mu-\delta$ for all $n\not\in \mathcal L$.
  
Fix $t \in (0,1]$. 
We define a multiplicative function $f=f_t: \mb{N} \ra \mb{U}$ at primes $p$ by
\begin{equation} \label{eq: ftPrime}
f_t(p) :=   g\bigg( \frac{ t \log p }{ 2\pi } \bigg)  \in \mb{U},
\end{equation}
and inductively on prime powers $p^m$, $m \geq 2$, via the convolution formula
\begin{align} \label{eq:convId}
f_t(p^m) :=   \frac 1{m} \sum_{1 \leq j \leq m} f_t( p^{ m-j } )  g\bigg( \frac{t \log p^j}{ 2\pi } \bigg).
\end{align}

Under these assumptions we will prove the following estimate:

\begin{theorem}\label{lem:asympForft}
Let $t\in [ -1 , 1]$ be such that $|t|$ is small but $|t| \gg (\log X)^{-\eps}$ for all $\eps > 0$. Then
$$
\sum_{n \leq X} \frac{f_t(n)}{n} = (1+O(|t|))  
\sum_{\ell \in \mathcal L} \frac{ X^{i\ell t} }{i  \ell' t} \frac{C_\ell   (it\log X)^{\gamma_\ell-1}}{\Gamma(\gamma_\ell)}
$$
where $C_\ell := \prod_{k \neq 0} k^{-g_{\ell-k}}$, and $\ell'=1$ if $\ell=0$ and $\ell'=\ell$ otherwise.
\end{theorem}

One can make the weaker assumption that $|g_n| \ll 1/|n|^{1+\epsilon}$ for all integers $n \neq 0$, and obtain the weaker, but satisfactory, error term 
$O(|t|^{\epsilon/2})$ in place of $O(|t|)$.

Henceforth fix $t$ and use $f=f_t$. By \eqref{eq:convId} and induction on $m \geq 1$ we have
$$
|f(p^m)| \leq \frac{1}{m } \sum_{1 \leq j \leq m} |f(p^{m-j})|  \leq 1,
$$
so that $f$ indeed takes values in $\mb{U}$.
If $F(s)$ is the Dirichlet series of $f$ for $\text{Re}(s) > 1$ then $F(s)$ is 
analytic and non-vanishing in that half-plane, and so $-\tfrac {F'}{F}(s)$ is also analytic there. 
The convolution identity \eqref{eq:convId} implies that 
\[
-\frac{F'}F(s) = \sum_{n\geq 1} g\bigg( \frac{t \log n}{ 2\pi } \bigg)\frac{\Lambda(n)}{n^s} .
\]

Integrating $-\tfrac {F'}F (s)$ termwise, we see that when $\text{Re}(s) > 1$,
$$
\log F(s) = \sum_{n \geq 1} \frac{ \Lambda_f(n) }{ n^s \log n } 
= \sum_{p^k} \frac{ g( \tfrac { t \log p^k }{ 2\pi } )}{kp^{ks}} = \sum_{m \in \mb{Z}}  g_m \sum_{p^k} \frac{ p^{ ikmt } }{ kp^{ ks } } 
= \sum_{ m \in \mb{Z} } g_m \log \zeta(s-imt),
$$
swapping orders of summation using the absolute summability of $\{g_m\}_m$. For $\text{Re}(s) > 1$, we may thus write \[
F(s) = \prod_{m \in \mb{Z}} \zeta(s-imt)^{g_m}.
\]
We will work with the finite truncations of this product,
\[
F_N(s) := \prod_{|m| \leq 2N} \zeta(s-imt)^{g_m} .
\]


The proof of Theorem \ref{lem:asympForft} relies on a technical contour integration argument  complicated by the possibility that the zeros and poles of $\zeta(s-imt)$ might contribute  essential singularities whenever $g_m \neq 0$. 
The following key technical  lemma will be proved in section \ref{sec: zetas}.

For given $\tau \in \mb{R}$ we define
\[
\sg(\tau) := \frac{c}{\log(2+|\tau|)},
\]
where $c > 0$ is chosen sufficiently small so that $\zeta(\sg+i\tau) \neq 0$ whenever $\sg \geq 1-\sg(\tau)$.

\begin{lemma}\label{lem: prodEst}
Let $t\in [ -1 , 1]$ be such that $|t|$ is small but $|t| \gg (\log X)^{-\eps}$ for all $\eps > 0$. Fix $A \geq 2$, let $N := \lceil \tfrac {(\log X)^A}{|t|} \rceil$ and  $T := ( N + \tfrac 12 ) |t|$. Also let $r_0 := \tfrac 14 \min\{\sg(3T), |t|\}$. 

\noindent (a) If $s = \sg + i\tau$ with $\sg \geq \frac{1}{\log X}$ and $|\tau| \leq T$ then
\[
F(s+1) = F_N(s+1) + O((\log X)^{-2}).
\]
(b) We have
\[
\max_{|\tau| \leq T} |F_N(1-r_0 + i\tau) |\ll_{\eps} (\log X)^{\eps}.
\]
(c) Let $\eta \in \{-1,+1\}$. Then
\[
\max_{-r_0 \leq \sg \leq r_0}  |F_N(1+\sg + i\eta T) |   \ll_{\eps} (\log X)^{\eps}.
\]
(d) If $|t|$ is sufficiently small then  for any $\ell \in \mb{Z}$,
\[
\prod_{ \ss{ |k| \leq 2N \\ k \neq \ell} } \zeta(1-i(k-\ell)t)^{g_k} = (1+O(|t|)) C_\ell(it)^{g_\ell-1}.
\]
More generally when $|n| \leq N$ and $|s| \leq 2 r_0$,
\[
\prod_{ \ss{ |k| \leq 2N \\ k \neq n } } |\zeta(1+s - i(k-n)t)^{g_k}| \ll_{\eps} (\log X)^{\eps}.
\]
\end{lemma}

\begin{proof}[Proof of Theorem \ref{lem:asympForft}]
Let $c_0 := \tfrac 1{\log X}, A=2$ and $N$ and $T$ be as in Lemma \ref{lem: prodEst} so that $T \geq (\log X)^2$. 
By a quantitative form of Perron's formula \cite[Cor. 2.4]{Ten}, we have
\[
\sum_{n \leq X} \frac{f(n)}{n} = \frac{1}{ 2\pi i } \int_{ (c_0) } F( s+1 ) \frac{ X^s }{s} ds
= \frac 1{ 2\pi i} \int_{ c_0 - iT }^{ c_0 + iT } F( s + 1 ) \frac{ X^{s} }{ s } ds 
+ O\bigg( \frac{1}{\log X} \bigg).
\]
 By Lemma \ref{lem: prodEst}(a),
\begin{align}\label{eq:truncPerr}
\sum_{n \leq X} \frac{f(n)}{n} 
= \frac 1{ 2\pi i} \int_{ c_0 - iT }^{ c_0 + iT } F_N( s + 1 ) \frac{ X^{s} }{ s } ds 
+ O\bigg( \frac{1}{\log X} \bigg).
\end{align}
We now deform the path $[c_0-iT, c_0+iT]$ into a contour intersecting with the critical strip within the common zero- and pole-free regions of $\{\zeta(s-int)\}_{|n| \leq 2N}$.
Since $|\text{Im}(s)-nt| \leq 3T$ we see that $\zeta(s+1-int) \neq 0$ for all $|n| \leq 2N$ and $|\text{Im}(s)| \leq T$ whenever $\text{Re}(s) \geq -\sg(3T)$. 

Let $\mc{H}$ denote the Hankel contour%
\footnote{That is, for $r := \frac{1}{\log X}$ the contour consisting of the circular segment $\{s \in \mb{C} : |s| = r, \text{arg}(s) \in (-\pi,\pi)\}$ (omitting the point $s = -r$) together with the lines $\text{Re}(s) \leq -r$ covered twice, once at argument $+\pi$ and once at argument $-\pi$, traversed counterclockwise.} \cite[p. 179]{Ten}
of radius $\tfrac 1{\log X}$, and let $r_0 := \tfrac 14 \min\{|t|,\sg(3T)\}$.
For each $|n| \leq N$ we write 
$$
\mc{H}_n := (\mc{H} + int) \cap \{\sg + i\tau \in \mb{C} : \sg \geq -r_0\}.
$$
%
%
We glue the paths $\{\mc{H}_n\}_{|n| \leq N}$ together and to the horizontal lines $[-r_0+iT,c_0 + iT]$ and $[-r_0-iT,c_0-iT]$ using the line segments 
\begin{align*}
L_n &:= (-r_0+in|t|, -r_0+i((n+1)|t|)) \text{ for } -N \leq n \leq N-1, \\
B_1 &:= [-r_0-iT, -r_0 -iN|t|), \quad\quad 
B_2 := (-r_0+iN|t|, -r_0+iT]
\end{align*}
Denote this concatenated path by $\Gamma_N$ and define the contour
$$
\Gamma := [c_0-iT,c_0+iT] \cup [c_0+iT,-r_0+iT] \cup \Gamma_N \cup [-r_0-iT,c_0-iT],
$$
traversed counterclockwise. 
Since $F_N(s+1)/s$ is analytic in the interior of the component cut out by $\Gamma$, the residue theorem implies that 
\begin{align}\label{eq:contour}
&\frac 1{ 2\pi i} \int_{ c_0 - iT }^{ c_0 + iT } F_N( s + 1 ) \frac{ X^{s} }{ s } ds = \mc{M} + \mc{R},
\end{align}
where $\mc{M} := \frac 1{2\pi i}  \sum_{ |n| \leq N }   \int_{ \mc{H}_n } \frac{ F_N(s+1) }{ s } X^s  ds$ is the contribution from the Hankel contours, and
\begin{align*}
\mc{R} &:= \frac 1{2\pi i}   \bigg(    \int_{B_1} \frac{ F_N(s+1) }{ s } X^s ds    +    \int_{B_2} \frac{ F_N(s+1) }{ s } X^s ds  +  \sum_{ -N \leq n \leq N-1 } \int_{ L_n } \frac{ F_N(s+1) }{ s } X^s ds   \bigg) \\
&- \frac 1{2\pi i}  \bigg( \int_{ -r_0 - iT }^{ c_0 - iT } \frac{ F_N(s+1) }{s} X^s  ds    -    \int_{ -r_0 + iT }^{ c_0 + iT } \frac{ F_N(s+1) }{ s } X^s ds   \bigg). 
\end{align*}
Along the segments $L_n$ and $B_j$, where $\text{Re}(s+1) = 1 - r_0$, we apply Lemma \ref{lem: prodEst}(b) to obtain
\begin{align*}
&\sum_{ -N \leq n \leq N-1 } \bigg| \int_{ L_n } \frac{ F_N(s+1) }{s} X^s ds \bigg| + \sum_{j = 1,2} \bigg| \int_{ B_j } \frac{ F_N(s+1) }{s} X^s ds \bigg|  \\
&\ll_{\eps} X^{ -r_0 } (\log X)^{ \eps } \bigg( \sum_{ |n| \leq N } \frac 1{r_0 + |nt|} + \frac{1}{T} \bigg) 
\ll \frac 1{\log X}.
\end{align*}
Along the horizontal segments we use Lemma \ref{lem: prodEst}(c) to give
\begin{align*}
\bigg|\int_{ -r_0 \pm iT }^{ c_0 \pm iT } \frac{ F_N(s+1) }{s} X^s ds \bigg| 
\ll_{\eps} \frac { (\log X)^{\eps} }T \ll \frac 1{\log X}.
\end{align*}
Thus, $\mc{R} \ll \tfrac 1{\log X}$,
 and it remains to treat $\mc{M}$. 
For each $|n| \leq N$ note that $\mc{H}_n = \mc{H}_0 + int$, and so by a change of variables,
\[
\mc{M}_n := \frac 1{ 2\pi  i } \int_{ \mc{H}_n } \frac{ F_N(s+1) }{s} X^s ds 
= \frac{ X^{ int } }{ 2\pi i } \int_{ \mc{H}_0 } \frac{ G_n(s) }{ s + int } \frac{ X^s }{ s^{ g_n } } ds,
\]
where we set
\[
G_n(s) := \left[ s \zeta( s+1 ) \right]^{ g_n } \prod_{ \ss{ |m| \leq 2N \\ m \neq n } } \zeta( s + 1 - i (m-n)t )^{ g_m }.
\] 
$G_n$ is analytic near $0$, and when $|s| \leq \tfrac 12 \min\{ |t| , \sg(3T) \} = 2r_0$ we can write
\[
G_0(s) = \sum_{ j \geq 0 } \mu_{0,j}(t) s^j, \quad\quad 
\frac{ G_n(s) }{ s + i nt } = \sum_{ j \geq 0 } \mu_{n,j}(t) s^j \ \  \text{ if } n \neq 0.
\]
The functions $\mu_{n,j}(t)$ are determined by Cauchy's integral formula as
\begin{equation}\label{eq:munj}
\mu_{n,j}(t) = \frac{1}{2\pi i} \int_{|s| = r} \frac{ G_n(s) }{ (s+int)^{ 1_{n \neq 0} } } \frac { ds }{ s^{j+1} }, \quad 0 < r \leq 2 r_0.
\end{equation}
Note in particular that
\[
\mu_{0,0}(t) = \prod_{ \ss{ |m| \leq 2N \\ m \neq 0 } } \zeta( 1-imt )^{ g_m }, \quad 
\mu_{n,0}(t) = \frac 1{ int } \prod_{ \ss{ |m| \leq 2N \\ m \neq n } } \zeta( 1 - i (m-n)t )^{ g_m } \text{ if } n \neq 0,
\]
while for $j \geq 1$ we take $r = 2r_0$ and apply Lemma \ref{lem: prodEst}(d) in \eqref{eq:munj}
to get\footnote{ Here and below, we repeatedly use the fact that if $z \in \mb{C} \backslash\{0\}$ then, choosing any appropriate branch of complex argument, we have 
$|z^{g_n}| = |z|^{\text{Re}(g_n)} \exp(-\text{Im}(g_n) \cdot \text{arg}(z)) \asymp |z|^{\text{Re}(g_n)}$ for all $n \in \mb{Z}$,
as $\sup_{n \in \mb{Z}} |g_n| \leq 1$ and $|\text{arg}(z)|$ is uniformly bounded.}
for all $|n| \leq N$,
\begin{align}\label{eq: muNj}
|\mu_{n,j}(t)| &\leq r^{-j} \max_{|s| = r}\frac{\prod_{\ss{|k| \leq 2N \\ k \neq n}} |\zeta(s+1-i(k-n)t)^{g_k}|}{|s+int|^{1_{n \neq 0}}} \ll_{\eps} \frac{ 2^j(\log X)^{\eps} }{ \min\{|t|,\sg(3T)\}^j (1_{n =0} +|nt|) }. 
\end{align}
Integrating over $\mc{H}_0$ (noting that $|s| \leq r/2$ for all $s \in \mc{H}_0$) and applying \cite[Cor. 0.18]{Ten}, when $n \neq 0$ we obtain
\begin{align*}
\mc{M}_n &= \frac{ X^{ int } \mu_{n,0}(t) }{ 2\pi i}  \int_{ \mc{H}_0 } X^s s^{-g_n}ds 
+ O_{\eps}\bigg( \frac{ (\log X)^{\eps} }{ |nt| \min\{ |t|, \sg(3T) \} }  \int_{ \mc{H}_0 } X^{ \text{Re}(s) } |s|^{ 1-\text{Re}(g_n) } |ds| \bigg)\\
&= \frac{ X^{int} \mu_{n,0}(t) }{ \Gamma( g_n ) } \bigg( (\log X)^{ g_n-1 } + O(X^{-r_0}) \bigg) \\
&+ O_{\eps} \bigg( \frac{ (\log X)^{2\eps} }{ |n|t } \bigg( \int_{ -r_0 }^{ -\tfrac 1{\log X} } |\sg|^{1-\text{Re}(g_n)} X^{\sg} d\sg 
+ (\log X)^{\text{Re}(g_n)-2}\bigg)\bigg)  \\
&= \frac{ X^{int} }{ int   } \frac{  (\log X)^{ \gamma_n-1 } }{ \Gamma(\gamma_n) } \bigg( \prod_{ \ss{ |k| \leq 2N \\ k \neq n } } \zeta( 1 - i(k-n)t )^{g_k} 
+ O_{\eps}\bigg( \frac 1{(\log X)^{1-\eps}} \bigg) \bigg)
\end{align*}
and similarly
\[
\mc{M}_0
= \frac{ (\log X)^{\gamma_0-1} }{ \Gamma(\gamma_0) } \bigg( \prod_{ \ss{ |k| \leq 2N \\ k \neq 0 } } \zeta(1-ikt)^{g_k} 
+ O_{\eps}\bigg( \frac 1{(\log X)^{1-\eps}} \bigg) \bigg).
\]
We next focus on the products of $\zeta$-values. 
When  $n = \ell \in \mc{L}$, Lemma \ref{lem: prodEst}(d) gives
\[
\prod_{ \ss{ |k| \leq 2N \\ k \neq \ell }} \zeta(1-i(k-\ell)t)^{g_k} = (1+O(|t|)) C_\ell(it)^{g_\ell-1}.
\]
We saw above that $\text{Re}(\gamma_n) \leq \mu - \delta$ for all $n\not\in \mathcal L$. 
Combining this with Lemma \ref{lem: prodEst}(d) and the estimates $|t|^{\text{Re}(g_\ell)-1} \geq 1$ (since $|g_\ell| \leq 1$ for all $\ell$) and $1/\Gamma(g_n) \ll 1$ uniformly (since $1/\Gamma$ is entire), \ when $\ell \notin \mc{L}$ we obtain
\begin{align*}
\sum_{ \ss{ |n| \leq N \\ n \neq \ell } } \prod_{ \ss{ |k| \leq 2N \\ k \neq n } }  | \zeta(1 - i(k-n)t )^{g_k}| \frac{ (\log X)^{\text{Re}(g_n)} }{ |\Gamma(g_n+1_{n = 0})| ( 1_{n = 0} + |nt|\log X ) } &\ll_{\eps} \frac{ (\log X)^{\mu-1-\delta + \eps} }{ |t| } \\
&\ll \frac{ (|t|\log X)^{\mu-1-\delta + \eps} }{ |t| }.
\end{align*}
Accounting for the error term for $\mc{M}_\ell$ and using $\delta < 1$, it follows that
\begin{align*}
\mc{M} &= \bigg(1+O(|t|)\bigg)  \sum_{\ell \in \mathcal L} \frac{ X^{i\ell t} }{i \ell' t}
\frac{C_\ell   (it\log X)^{\gamma_{\ell}-1}}{\Gamma(\gamma_\ell)}+ O_{\eps} \bigg( \frac{ (|t|\log X)^{\mu-1 + \eps} }{ |t| } \cdot \frac 1{ (\log X)^{\delta} } \bigg).
\end{align*}
Setting $\delta' := \tfrac 12 \delta$ and taking $\eps < \delta'$ we obtain
\begin{align*}
\mc{M} &= \bigg(1+ O\bigg( |t| + \frac{1}{(\log X)^{\delta'}} \bigg)\bigg)   \sum_{\ell \in \mathcal L} \frac{ X^{i\ell t} }{i\ell'  t}
\frac{C_\ell   (it\log X)^{\gamma_\ell-1}}{\Gamma(\gamma_\ell)} .
\end{align*}
The proof is completed upon combining this estimate with our bound for $\mc{R}$ in \eqref{eq:contour}, and then using \eqref{eq:truncPerr}.
\end{proof} 

The values of $f_t(p)$ at primes $p$ are crucial in obtaining the shape of the asymptotic formula in Theorem \ref{lem:asympForft}, not the 
 values $f_t(p^m)$ with $m \geq 2$  at prime powers, as the following Corollary shows:

\begin{corollary}\label{cor:genPrimPow}
Assume the hypotheses of Theorem \ref{lem:asympForft}. Let $f : \mb{N} \ra \mb{U}$ be a multiplicative function such that $f(p) = f_t(p)$ for all primes $p$, and define a multiplicative function $h$ so that $f := f_t \ast h$. 
Then
$$
\sum_{n \leq X} \frac{f(n)}{n} =  \sum_{\ell \in \mathcal L}  \frac{ X^{i\ell t} }{i\ell'  t} \frac{C_\ell   (it\log X)^{\gamma_\ell-1}}{\Gamma(\gamma_\ell)}(H(1+i \ell t) +O(|t|) + o(1))        \text{ where } H(s):=\sum_{n\geq 1} \frac{h(n)}{n^s};
$$
 moreover, for each $\ell \in \mc{L}$ we have $H(1+i \ell t)\ne 0$ unless $f(2^k)=-2^{ik \ell t}$ for all $k\geq 1$.
\end{corollary}

One expects that $\mathcal L$ typically contains just one element, $\{ \ell\}$,  and so an asymptotic is given by this formula for all large $X$ if
 $H(1+i \ell t)\ne 0$ 
(that is, if $f(2^k)\ne -2^{ik \ell t}$ for some $k\geq 1$).
If $\mathcal L$   contains more than one element first note that $H(1+i \ell t)=0$ for at most one value of $\ell$, so we have a sum of main terms of similar magnitude.  For $X\in [Z,Z^{1+o(1)}]$ we get a formula of the form $ ( \sum_{\ell \in \mathcal L} c_\ell X^{i\ell t} +o(1) )(\log Z)^{\mu -1}$ (where the $c_\ell$ depend on $Z$ but not $X$) and such a finite length trigonometric polynomial 
will have size $o(1)$ for a logarithmic measure $0$ set of $X$-values (that is for $X\in [Z,YZ]$ where $|t|\log Y\to \infty$ with $\log Y=o(\log Z)$).

\begin{proof} 
Each $h(p)=0$ so that $h(n)=0$ unless $n$ is powerful. Since each $|f_t(p^k)|, |f(p^k)|\leq 1$, we deduce by induction that $|h(p^k)|\leq 2^{k-1}$ for each $k\geq 2$. We begin by assuming that each $h(2^k)=h(3^k)=0$, and so if $(n,6)=1$ then $|h(n)|\leq n^\kappa$ where $\kappa=\frac{\log 2}{\log 5}$($<\frac 12$).

As $h(n)=0$ unless $n$ is powerful and $(n,6)=1$, we have
\[
\sum_{n>N } \frac{|h(n)|}n \ll N^{\kappa-\frac 12} \text{ and } \sum_{b\geq 1}\frac{ |h(b)|\log b} b<\infty.
\]
Select $A>0$ so that $A(\frac 12-\kappa)>2-\mu$. Then with $M := (\log X)^A$ we have $\sum_{b>M} \frac{ |h(b)| }{b} =o((\log X)^{\mu-2})$, and by Theorem \ref{lem:asympForft},

 \begin{align*}
\sum_{ n \leq X } \frac{ f(n) }{n} 
&= \sum_{ \ss{ ab \leq X \\ b \leq M } } \frac{ f_t(a) h(b) }{ ab } + \sum_{ \ss{ ab \leq X \\ b > M }} \frac{ f_t(a) h(b) }{ ab } \\
&= \sum_{ b \leq M } \frac{ h(b) }{b} \sum_{ a \leq X/b } \frac{ f_t(a) }{a} + O\bigg( \sum_{ a \leq X } \frac{1}{a} \sum_{ M < b \leq X/a } \frac{ |h(b)| }{b} \bigg) \\
&= ( 1 + O( |t| ) ) \sum_{\ell \in \mathcal L} \frac{ X^{i\ell t} }{i \ell'  t} \frac{C_\ell   (it\log X)^{ \gamma_\ell - 1}}{\Gamma(\gamma_\ell)}
  \sum_{ b \leq M } \frac{ h(b) }{ b^{1+i \ell t} } \bigg( 1 - \frac{ \log b }{ \log X } \bigg)^{\gamma_\ell-1} \\
&+  o((\log X)^{\mu-1}).
\end{align*}
The claimed formula follows since  
 \[
\sum_{ b \leq M } \frac{ h(b) }{ b^{1+i \ell t} } \bigg( 1 - \frac{ \log b }{ \log X } \bigg)^{\gamma_\ell - 1} = 
H(1+i \ell t) +O\bigg( \sum_{ b > M } \frac{ |h(b)| }{ b } + \frac{ 1}{ \log X } \sum_{ b \leq M } \frac{ |h(b)|\log b}{ b } \bigg)    = 
H(1+i \ell t) +o(1).
\]
Now suppose that $h(3^k)$ is not necessarily $0$.  The key issue is
\[
\sum_{n>N } \frac{|h(n)|}n = \sum_{k\geq 0} \frac{|h(3^k)|}{3^k} \sum_{\substack{n>N/3^k \\ (n,6)=1}} \frac{|h(n)|}n
\ll \sum_{k\geq 0} \frac{|h(3^k)|}{3^k} (N/3^k)^{\kappa-\frac 12}\leq N^{\kappa-\frac 12}\sum_{k\geq 0}   \frac{2^{k-1}}{(3^{\kappa+\frac 12})^k}\ll N^{\kappa-\frac 12}
\]
since $3^{\kappa+\frac 12}>2$.

Our assumptions guarantee that the sum for $H(s)$ converges on the $1$-line.  Is   $H(1+i \tau)\ne 0$ for $\tau \in \mb{R}$?  We see that the Euler factors converge on the $1$-line
and indeed  
\[
\bigg| \sum_{k\geq 0} \frac{h(p^k)}{p^{k(1+i \tau)}} \bigg| \geq 1-\sum_{k\geq 1} \frac{|h(p^k)|}{p^k}\geq 1-\sum_{k\geq 2} \frac{2^{k-1}}{p^k}=1-\frac{2}{p(p-2)}>0 
\]
for each prime $p\geq 3$.

Now suppose that $h(2^k)$ is not necessarily $0$.  The analogous argument works for any $h(.)$ for which there exists $\epsilon>0$ such that $|h(2^k)| \ll (2^k)^{1-\epsilon}$.  To establish this we first assume that each
$|f_t(2^k)|=1$ so write $f_t(2^k)=e(\theta_k)$  and $g(\frac{t \log 2^j}{ 2\pi }) = r_je(\gamma_j)$ with $0 \leq r_j \leq 1$. Then
\eqref{eq:convId} becomes $m\, e(\theta_m)  = \sum_{1 \leq j \leq m} r_je(\theta_{m-j}+\gamma_j)$. This implies that $r_j = 1$ and
$\theta_m=\theta_{m-j}+\gamma_j \ \mod 1$ for $1\leq j\leq m$.
Now $\theta_0=\gamma_0=0$ and so $\theta_m=\gamma_m=m\gamma_1\ \mod 1$. But then
$\sum_{k\geq 0} f_t(2^k)/2^{ks} = \sum_{k\geq 0} (e(\gamma_1)/2^s)^k = (1-e(\gamma_1)/2^s)^{-1}$ and so if $k\geq 1$ then
$h(2^k)=f(2^k)-e(\gamma_1)f(2^{k-1})$ and so each $|h(2^k)|\leq 2$.

Otherwise there exists a minimal $k\geq 1$ such that $|f_t(2^k)|<1$; let $\delta=1-|f_t(2^k)|\in (0,1]$.  Now select $\alpha>0$ for which
$\delta(\alpha-1)=\alpha^k(2-\alpha)$ so that $1<\alpha<2$. We claim that  
$|h(2^m)|\leq \kappa \alpha^m$ for all $m\geq 0$, where $\kappa:=\max_{0\leq m\leq k} |h(2^m)|/\alpha^m$. This is trivially true for $m\leq k$; otherwise for $m>k$ we have (as $h(1)=1, h(2)=0$)
\begin{align*}
|h(2^m)| &= \bigg| f(2^m)- \sum_{j=0}^{m-1} f_t(2^{m-j})h(2^j)\bigg|
\leq 2+\sum_{\substack{j=2 \\ j\ne m-k}}^{m-1} |h(2^j)| +(1-\delta) |h(2^{m-k})| \\
&\leq  \sum_{\substack{j=0 \\ j\ne m-k}}^{m-1} \kappa \alpha^j +(1-\delta) \kappa \alpha^{m-k}
< \kappa\alpha^m\bigg( \sum_{i\geq 1} \alpha^{-i} - \delta  \alpha^{-k}\bigg)=\kappa\alpha^m
\end{align*}
as $2\leq \kappa+\kappa \alpha$, by induction,  using the definition of $\alpha$.

Finally we wish to determine whether the Euler factor of  $H(1+i\ell t)$ at $2$ equals $0$.  This equals the 
Euler factor for $f$ at $1$ divided by the Euler factor for $f_t$ at $1$.  
Since each $|f_t(2^k)|\leq 1$ the denominator is bounded; since $  |f(2^k)|\leq 1$
we have  $\sum_{k\geq 0} \frac{f(2^k)}{2^{k(1+i\ell t)}}=0$ if and only if $f(2^k)=-2^{ik\ell t}$ for all $k\geq 1$.

\end{proof}

\subsection{Our specific example}
We now use Theorem \ref{lem:asympForft} to construct a multiplicative function $f$ satisfying the conclusion of Proposition \ref{prop:lambdaOpt}(b). 
We will use the auxiliary 1-periodic function
$$
g(u) = \frac{e(u)-\lambda}{|e(u)-\lambda|} = \frac{|e(u)-\lambda|}{e(-u)-\lambda} = \sum_{n \in \mb{Z}} g_ne(nu).
$$
We see that $g$ takes values on $S^1$ with $g(0) = 1$. We will verify the following properties of $\{g_n\}_n$ in the appendix (section \ref{sec:Fourier}), which shows that $g$ satisfies the assumptions required to apply Theorem \ref{lem:asympForft}.

\begin{lemma} \label{lem:FourierCoeffs}
For the $\{g_n\}_n$ defined just above we have:
\begin{enumerate}[(a)]
\item $g_n \in \mb{R}$ for all $n$,
\item $|g_n| \ll \ ( \tfrac{ 2\lambda }{ 1+\lambda^2 } )^{ |n| } \leq 0.99^{|n|}$ for all $n \in \mb{Z}$,
\item $g_{-n} < g_n$ for all $n \geq 1$, 
\item $g_1 = 0.7994 \dots$, and there is $\delta > 0$ such that $g_n \le g_1 -\delta$ or all $n \neq 0,1$, and
\item $g_0 < g_1-1$.
\end{enumerate}
\end{lemma}

 \begin{proof}[Deduction of Proposition \ref{prop:lambdaOpt}(b)]
Let $x$ be large. Let $t \in [\tfrac 1{ \log\log x } , 1]$ be small, and set $y_t := e^{\tfrac 1t}$ and $f = f_t$.
For small enough $t$, Theorem \ref{lem:asympForft} yields
\begin{equation}\label{eq:asympEstft}
\bigg| \sum_{ n \leq x } \frac{ f(n) }{n} \bigg| \asymp t^{-1}( t \log x )^{ g_1-1 } 
\asymp \log x \exp\bigg( (g_1-2) \sum_{ y_t < p \leq x } \frac 1p \bigg).
\end{equation}
Using the definition of $\lambda$,
\begin{align*}
2 - \lambda = \int_0^1 \frac{ |e(u)-\lambda| }{ e(-u) - \lambda } ( e(-u) - \lambda ) du 
= \int_0^1 g(u) e(-u) du - \lambda \int_0^1 g(u) du = g_1 - \lambda g_0,
\end{align*}
so that $g_1-2=\lambda (g_0-1)$. By partial summation and the prime number theorem we have 
\begin{align*}
\mb{D}(f,1;y_t,x)^2 = \sum_{ y_t < p \leq x }  \frac{ 1-\text{Re}( g( \tfrac t{2\pi} \log p) ) }{p} 
&= \text{Re}\bigg( \int_0^{1} \bigg( 1-g(u) \bigg) du \bigg) \log\bigg( \frac{ \log x }{ \log y_t } \bigg) + O(1) \\
&= (1-g_0) \sum_{y_t < p \leq x} \frac{1}{p} + O(1).
\end{align*}
Combining these last few observations we deduce that
\[
\bigg| \sum_{ n \leq x } \frac{ f(n) }n \bigg| 
\asymp \log x \exp\bigg( \lambda (g_0-1) \sum_{ y_t < p \leq x } \frac 1p \bigg) 
\asymp \log x \exp\bigg( -\lambda \ \mb{D}(f,1;y_t,x)^2 \bigg).
\]

Finally, since $g$ is Lipschitz and $g(0) = 1$,
\[
|g(\tfrac t{2\pi} \log p) - 1| \ll t\log p \text{ for all } p \leq y_t, 
\]
and so by Mertens' theorem,
\[
\mb{D}(f,1;y_t)^2 = \sum_{ p \leq y_t } \frac{ 1 - \text{Re}( g( \tfrac t{2\pi} \log p ) ) }{p} 
\ll t \sum_{ p \leq y_t } \frac{ \log p }{p} \ll 1. 
\]
Combining these last two estimates, \eqref{eq:GS25} follows. 
\end{proof}


\appendix 
\section{Auxiliary results towards Proposition \ref{prop:lambdaOpt}(b)}
We establish the technical lemmas \ref{lem:FourierCoeffs} and \ref{lem: prodEst}, used in the proofs of 
Proposition \ref{prop:lambdaOpt}(b) and Theorem \ref{lem:asympForft}.

\subsection{On the Fourier Coefficients of $g$} \label{sec:Fourier}
 \begin{proof}[Proof of Lemma \ref{lem:FourierCoeffs}]
(a) Since $g(u) = \bar{g}(-u)$, a term-by-term comparison of the Fourier series of each shows that $g_n = \bar{g}_n$ and thus $g_n \in \mb{R}$ for each $n \in \mb{Z}$.

\noindent (b) We will prove a numerically sharper bound, which will be used in part (d). 
Note that
$$
g(u)= \frac{ e(u) - \lambda }{ \sqrt{ 1+\lambda^2 } } \cdot \bigg( 1 - \frac{ \lambda }{ 1+\lambda^2 }( e(u) + e(-u) ) \bigg)^{ -1/2 }. 
$$
As $2\lambda < 0.99(1+\lambda^2)$, Taylor expanding the bracketed expression gives
\begin{align*}
g(u) &= \frac{ e(u) - \lambda }{ \sqrt{ 1+\lambda^2 } } \sum_{ j \geq 0 } \binom{ -1/2 }{j} (-1)^j \bigg( \frac{\lambda}{1+\lambda^2} \bigg)^j \sum_{ 0 \leq i \leq j } \binom{j}{i} e( (2i-j) u) \\
&= \frac{ e(u) - \lambda }{ \sqrt{ 1+\lambda^2 } } \sum_{ n \in \mb{Z} } e(nu) \sum_{ \ss{ j \geq |n| \\ 2|(n+j)} } \binom{j}{ (j+n)/2 } 2^{ -2j } \binom{ 2j }{ j } \bigg( \frac{ \lambda }{ 1+\lambda^2 } \bigg)^j \\
&= \sum_{ m \in \mb{Z} } e(mu) \frac{ h_{m-1} - \lambda h_m }{ \sqrt{ 1+\lambda^2 } },
\end{align*} 
where we have set
\begin{align} \label{eq: hnDef}
h_n := \sum_{ \ss{ j \geq |n| \\ 2 | (j+n) } }  2^{ -2j } \binom{2j}{j} \binom{j}{ (j+|n|)/2 } \bigg( \frac{ \lambda }{ 1+\lambda^2 } \bigg)^j 
= \sum_{l \geq 0} \binom{ |n| + 2l }{ |n|+l } \binom{ 2(|n|+2l) }{ |n|+2l } \bigg( \frac{ \lambda }{ 4(1+\lambda^2) } \bigg)^{ |n|+2l }.
\end{align}
It follows that
$$
g_n = \frac{ h_{n-1}-\lambda h_n }{ \sqrt{ 1+\lambda^2 } } \text{ for all } n \in \mb{Z}. 
$$
We use Stirling's approximation in the form $\sqrt{2\pi n} (n/e)^n \leq n! \leq e^{ \tfrac 1{12} } \sqrt{2\pi n} (n/e)^n$ for $n \in \mb{N}$ (see \cite{Rob}), and bound $\binom{j}{(j+|n|)/2}$ by a central binomial coefficient as $\binom{j}{(j+|n|)/2} \leq \frac{1}{2^{\nu}}\binom{j+\nu}{(j+\nu)/2}$, where $\nu = 1_{2\nmid j}$. When $n \neq 0$ this gives
\[
2^{-2j}\binom{2j}{j} \binom{j}{(j+|n|)/2} \leq  2^{-2j} \bigg( \tfrac { e^{ \tfrac 1{12} } }{ \sqrt{ \pi j } } 2^{2j} \bigg) \bigg( \tfrac { e^{ \tfrac 1{12} } } { 2^{\nu} \sqrt{ \pi (j+\nu)/2} } 2^{j + \nu } \bigg) \leq \frac{ e^{ \tfrac 16 } \sqrt{2} }{ \pi j } 2^j, \text{ if } 2|(j+n).
\] 
Setting $c := \tfrac { 2 \lambda }{ 1+\lambda^2}$, we find using the first expression for $h_n$ in \eqref{eq: hnDef} that
\[
0 \leq h_m \leq \sum_{ \ss{ j \geq |m| \\ 2|(j+m) } } \frac { e^{ \tfrac 16 } \sqrt{2} }{ \pi j }  \, 2^j \bigg( \frac{ \lambda }{ 1+\lambda^2 }\bigg)^j 
\leq \frac { e^{ \tfrac 16 } \sqrt{2} }{ ( 1-c^2 ) \pi } \, \frac{ c^{ |m| } }{  |m|  }, \quad m \neq 0.
\]
Using the second expression for $h_n$ in \eqref{eq: hnDef}, when $|n| \geq 1$ we get
\begin{align*}
h_{n-1} - \lambda h_n &= \sum_{ l \geq 0 } \binom{ |n|-1+2l }{ |n|-1+l } \binom{ 2(|n|-1+2l) }{ |n|-1+2l } \bigg( \frac{ \lambda }{ 4(1+\lambda^2) } \bigg)^{ |n|-1+2l } \bigg( 1 - \frac{ 2|n|-1+4l }{ 2(|n|+l) } \frac{ \lambda^2 }{ 1 + \lambda^2 } \bigg) \\
&\leq \bigg(1-\frac{2|n|-1}{2|n|} \frac{\lambda^2}{1+\lambda^2} \bigg) h_{n-1}.
\end{align*}
Together with the previous bound, it follows that when $|n|\geq 2$,
\begin{align} \label{eq: expDecgn}
|g_n| \leq \frac{ e^{ \tfrac 16 }\sqrt{2} }{ (1-c^2 ) \pi \sqrt{  1+\lambda^2  } } \bigg(1-\frac{2|n|-1}{2|n|} \frac{\lambda^2}{1+\lambda^2} \bigg) \frac{ c^{|n| - 1} }{  | n | - 1  },
\end{align}
and the bound $|g_n| \ll c^{|n|}$ immediately follows. 

\noindent (c) We deduce from \eqref{eq: hnDef} that $h_n = h_{-n}$ for any $n \geq 1$. Thus,
\begin{equation} \label{eq: gnPlusMinus}
\sqrt{ 1+\lambda^2 } g_n 
= \begin{cases} 
h_{ n-1 } - \lambda h_n  &\text{ if } n \geq 1,\\ 
h_{ |n|+1 } - \lambda h_{ |n| }   &\text{ if } n \leq 0.
\end{cases}
\end{equation}
To prove $g_{-n} < g_n$ for all $n \geq 1$ we need only show that $h_{n+1} < h_{n-1}$ for all $n \geq 1$. \\
To see this, note that
\begin{align*}
h_{n+1} &= \bigg( \frac{ \lambda }{ 2(1+\lambda^2) } \bigg)^2 \sum_{ l \geq 0 } \binom{ n-1+2l }{ n-1+l } \binom{ 2(n-1+2l) }{ n-1+2l } \bigg( \frac{ \lambda }{ 4(1+\lambda^2) } \bigg)^{ n-1+2l } \frac{ (2n+4l)^2-1 }{ (n+l+1)(n+l) } \\
&\leq 16\bigg( \frac{ \lambda }{ 2 (1+\lambda^2) } \bigg)^2 
\sum_{ l \geq 0 } \binom{ n-1+2l }{ n-1+l } \binom{ 2(n-1+2l) }{ n-1+2l } \bigg( \frac{ \lambda }{ 4(1+\lambda^2) } \bigg)^{ n-1+2l },
\end{align*}
and thus 
$h_{n+1} \leq \bigg(\frac{2\lambda}{1+\lambda^2}\bigg)^2h_{n-1} < h_{n-1}$, 
as required. 

\noindent (d) Since $g_{-1} < g_1$, and $g_{-n} < g_n$  for $n \geq 2$ by (c), it is enough to show that there is $\delta > 0$ such that $g_n \leq g_1-\delta$ for all $n \geq 2$. 

The upper bound \eqref{eq: expDecgn} is decreasing with $n$, and one may verify that it gives $\leq 0.7 < g_1 - \tfrac 1{20}$ for $n = 10$. 
Thus, we clearly have $g_n \leq g_1 - \tfrac 1{20}$ for all $n \geq 10$. 

On the other hand, a computer-assisted calculation shows that $g_n \leq g_1 - \tfrac 14$, say, for all $2 \leq n \leq 9$:
\begin{align*}
&g_1 = 0.7994..., \quad g_2 = 0.2848..., \quad g_3 = 0.1659... \quad g_4 = 0.1102..., \quad g_5 = 0.0778..., \quad \\
&g_6 = 0.0568..., \quad g_7 = 0.0423..., \quad g_8 = 0.0321..., \quad g_9 = 0.0246....
\end{align*}
The claim thus follows with $\delta = \tfrac 1{20}$. 

\noindent (e)  Though this may be verified by computer, we give a proof. From \eqref{eq: gnPlusMinus},
$$
g_1 - 1 - g_0 = \frac 1{ \sqrt{ 1+\lambda^2 } } (h_0 - \lambda h_1 - (h_1 - \lambda h_0) ) - 1 = \frac{ 1+\lambda }{ \sqrt{1+\lambda^2} } (h_0-h_1) - 1.
$$
It is enough to show that $h_0 - h_1 > \frac{ \sqrt{1+\lambda^2} }{ 1+\lambda }$. 
To see this, we write
$$
h_0 - h_1 = \sum_{ l \geq 0 } 2^{ -2l } \binom{ 2l }{l} \binom{ 4l }{ 2l } \bigg( \frac{ \lambda }{ 1+\lambda^2 } \bigg)^l \bigg( 1 - \frac{ \lambda }{ 2(1+\lambda^2) } \frac{ 4l+1 }{ l+1 } \bigg) \geq 1-\frac{ \lambda }{ 2(1+\lambda^2) },
$$
using $\tfrac {\lambda}{2(1+\lambda^2)} \frac{4l+1}{l+1} \leq \tfrac {2\lambda}{1+\lambda^2} < 1$ and positivity to restrict to the term $l = 0$. But we see that
$$
1-\frac{\lambda}{2(1+\lambda^2)} > \bigg(1-\frac{2\lambda}{(1+\lambda)^2}\bigg)^{1/2} = \frac{\sqrt{1+\lambda^2}}{1+\lambda},
$$
since, setting $t = \frac{2\lambda}{(1+\lambda)^2}$, we have
$$
1-\frac{\lambda}{2(1+\lambda^2)} = 1-\frac{t}{4} \cdot \bigg(1+\frac{2\lambda}{1+\lambda^2}\bigg) \geq 1-\frac{t}{2} > (1-t)^{1/2},
$$
as required.
\end{proof}

\subsection{On products of shifted zeta functions} \label{sec: zetas}
\begin{proof}[Proof of Lemma \ref{lem: prodEst}] Throughout, set $G := \sum_{n \in \mb{Z}} |g_n| < \infty$ (since we assumed that $|g_n|\ll 1/(1+|n|)^3$) and $F_N(s) := \prod_{|n| \leq 2N} \zeta(s-int)^{g_n}$.

\noindent (a) Let $\sg \geq \tfrac 1{\log X}$. Since $|f(n)| \leq 1$, 
\[
|F(1+\sg+i\tau)| \leq \zeta(1+\sg) \ll \log X \text{ for all } |\tau| \leq T.
\]
When $|n| > 2N$, $|\tau-nt| \geq (2N+1)|t| - T \geq T$, so that also
\[
\zeta(1 + \sg + i(\tau-nt)) \ll \log(2+|\tau-nt|) \ll \log(2|n|) \text{ for all } |\tau| \leq T.
\]
Now $|\zeta(2+2\sigma)/\zeta(1+\sigma)|<|\zeta(1+\sg+i(\tau-nt))|<|\zeta(1+\sigma)|$ and so if $|g_n|<\frac 1{\log\log X}
$ then  
\begin{align}\label{eq: TaylorExp}
\zeta(1+\sg+i(\tau-nt))^{-g_n} = 1 + O(g_n \log\log X).    
\end{align}
This holds for $|n| > 2N$ as we assumed that $|g_n| \ll |n|^{-3}$.
Since $N \geq (\log X)^A$ for some $A \geq 2$ it follows that
\begin{align}
&\max_{|\tau| \leq T}| F_N(1+\sg+i\tau) - F(1+\sg+i\tau) | \nonumber \\
&\ll \max_{|\tau| \leq T} \bigg( |F(1+\sg+i\tau)| \cdot \bigg| \exp \bigg( \sum_{ |n| > 2N } \log \bigg| \zeta( 1+\sg + i (\tau-nt) )^{-g_n} \bigg| \bigg) - 1 \bigg| \bigg) \nonumber \\
&\ll (\log X) \sum_{n > 2N} |n|^{-3} \log \log (2n) \ll_{\epsilon} (\log X)N^{-2+\epsilon} \ll (\log X)^{-2}, \label{eq:FTrunc}
\end{align}
as required.

\noindent (b) Let $|\tau| \leq T$ and $|n| \leq 2N$, and put $\sg := 1-r_0$.
If $|\sg-1+ i(\tau-nt)| \leq 1$ then 
\[
|\zeta( \sg + i(\tau-nt))|^{\pm 1} \ll |\sg-1 + i(\tau-nt)|^{-1} \leq r_0^{-1}.
\]
Otherwise, $|\zeta( \sg + i(\tau-nt) ) |^{\pm 1} \ll \log(2+|\tau-nt|) \ll \log T$.
Thus, for any $|\tau| \leq T$,
\begin{align*}
| F_N( \sg + i\tau) |
\ll \prod_{ \ss{ |n| \leq 2N \\ |-r_0 + i(\tau-nt)| \leq 1 } } r_0^{-|g_n|} \cdot \prod_{ \ss{ |n| \leq 2N \\ |-r_0 + i(\tau-nt)| > 1 } } (\log T)^{|g_n|} 
\ll (r_0^{-1} \log T)^G.
\end{align*}
Since $\min\{\sg(3T),|t|\}^{-1} \log T \ll_{\eps} (\log X)^{\eps}$, the claim follows.

\noindent (c) Assume $\eta = +1$; the claim with $\eta = -1$ is completely analogous.
Note that $|T-kt| \geq |t|/2$ for all $|k| \leq 2N$. If $|\sg + i(T-kt)| \leq 1$ then $|kt| \geq T/2\geq N|t|/2$, i.e., $|k| \geq N/2$, and then
\[
|\zeta(1+\sg+i(T-kt))^{g_k}| \ll |T-kt|^{-|g_k|} \ll |t|^{-|g_k|}.
\]
Splitting the product as in (b) and using $|g_n|\ll 1/(1+|n|)^3$ for each $|\sg| \leq r_0$ we get
\begin{align*}
|F_N( 1 + \sg + iT ) |
&\ll \prod_{ |n| \geq \tfrac N2 } |t|^{-|g_n|} \cdot \prod_{ \ss{ |n| \leq 2N \\ |\sg + i(T-nt)| > 1 } } (\log T)^{|g_n|} \\
&\ll  \exp\bigg( \log( \tfrac 1{|t|} ) \sum_{ |n| \geq \tfrac N2 } |n|^{-3} \bigg) (\log T)^G \ll (\log \log x)^G,
\end{align*}
and the claim follows.

\noindent (d) Observe that if $|m| \leq |t|^{-1}$ then $imt \zeta(1+imt) = 1+O(|mt|)$; 
otherwise, if $|mt| \geq 1$ then when $|t|$ is small enough $m$ is large and we may Taylor expand 
\[
[imt\zeta(1+imt)]^{g_{\ell-m}} = 1 + O\bigg( |g_{\ell-m}| \log (2+|mt|) \bigg),
\]
similarly to \eqref{eq: TaylorExp}. 
Thus, since $|g_{\ell-m}| \ll (1+|m-\ell|)^{-3}$   we have
\begin{align*}
&\bigg| \prod_{m \neq 0} [ imt \zeta(1+imt) ]^{ g_{ \ell-m } } - 1 \bigg|  \\
&\ll |t|\sum_{ |m| \leq |t|^{-1} } |g_{\ell-m}||m| + \sum_{ |m| > |t|^{-1} } |g_{\ell-m}| \log( 2+|mt| )  \\
&\ll |t|\sum_{|m| \leq |t|^{-1}} (1+|m-\ell|)^{-2} + \sum_{|m| > |t|^{-1}} \frac{ \log(2+|m|) }{ (1 + |m-\ell| )^3 } 
\ll |t|.
\end{align*}
Since also $|t| \geq \tfrac 1{\log X},$ it follows that (handling the range $|k| > 2N$ as in (a))
\begin{align*}
\prod_{ \ss{ |k| \leq 2N \\ k \neq \ell } } \zeta( 1- i(k-\ell)t )^{g_k} 
&= \bigg( 1+ O\bigg( \frac{1}{ (\log X)^{2} } \bigg) \bigg) \prod_{ m \neq 0 } (imt)^{ -g_{ \ell-m } } \prod_{ n \neq 0 } [ int \zeta(1+int) ]^{ g_{ \ell-n }} \\
&= \bigg( 1 + O( |t| ) \bigg) C (it)^{g_\ell-1}
\end{align*}
since $\sum_{m \in \mb{Z}} g_m = g(0) = 1$, and the first claim is proved. \\
Suppose more generally that $|n| \leq N$, and that $|s| \leq \tfrac 12 \min\{ |t|, \sg(3T) \} = 2r_0$. 
Whenever $k \neq n$ we have $|s-i(k-n)t| \geq |t|/2$.
Thus, arguing as in (c),
\[
\prod_{ \ss{ |k| \leq 2N \\ k \neq n } } |\zeta(s +1 - i(k-n)t)^{g_k}| 
\ll \prod_{ \ss{ |k| \leq 2N \\ k \neq n \\ |s-i(k-n)t| \leq 1 } } |t|^{-|g_k|} \prod_{ \ss{ |k| \leq 2N \\ k \neq n \\ |s-i(k-n)t| > 1 } } (\log T)^{|g_k|} 
\ll (|t|^{-1} \log T)^G,
\]
and since $|t| \gg (\log X)^{-\eps}$ for any $\eps > 0$ the claim follows.
\end{proof}

\bibliographystyle{plain}

\end{document}